\documentclass[12pt]{amsart}

\usepackage{amssymb,amsmath,graphicx,color,textcomp, amsthm,bbm,bbold, enumerate,booktabs}
\usepackage{chngcntr}
\usepackage{apptools}
\usepackage{mathrsfs}
\usepackage{tikz}
\usetikzlibrary{trees}

\AtAppendix{\counterwithin{theorem}{section}}
\definecolor{Red}{cmyk}{0,1,1,0}

\definecolor{verde}{cmyk}{1,0,1,0}

\definecolor{loka}{cmyk}{.5,0,1,.5}
\definecolor{azul}{cmyk}{1,1,0,0}

\numberwithin{equation}{section}

\newcommand{\be}{\begin{equation}}
\newcommand{\ee}{\end{equation}}

\newtheorem{theorem}{Theorem}

\newtheorem{definition}{Definition}
\newtheorem{lemma}{Lemma}
\newtheorem{remark}{Remark}

\begin{document}
\title{On the $\mathbf{\rm\Psi}-$fractional integral and applications}
\author{J. Vanterler da C. Sousa$^1$}
\address{$^1$ Department of Applied Mathematics, Institute of Mathematics,
 Statistics and Scientific Computation, University of Campinas --
UNICAMP, rua S\'ergio Buarque de Holanda 651,
13083--859, Campinas SP, Brazil\newline
e-mail: {\itshape \texttt{ra160908@ime.unicamp.br, capelas@ime.unicamp.br }}}

\author{E. Capelas de Oliveira$^1$}

\begin{abstract}Motivated by the ${\rm \Psi}$-Riemann-Liouville $({\rm \Psi-RL})$ fractional derivative and by the ${\rm \Psi}$-Hilfer $({\rm \Psi-H})$ fractional derivative, we introduced a new fractional operator the so-called $\rm\Psi-$fractional integral. We present some important results by means of theorems and in particular, that the $\rm\Psi-$fractional integration operator is limited. In this sense, we discuss some examples, in particular, involving the Mittag-Leffler $({\rm M-L})$ function, of paramount importance in the solution of population growth problem, as approached. On the other hand, we realize a brief discussion on the uniqueness of nonlinear $\Psi$-fractional Volterra integral equation (${\rm VIE}$) using $\beta-$distance functions.

\vskip.5cm
\noindent
\emph{Keywords}:Fractional calculus, $\rm\Psi-$fractional integral, population growth model, fractional Volterra integral equation, $\beta-$distance.
\newline 
MSC 2010 subject classifications. 26A33; 36A08; 34A12; 34A40.
\end{abstract}
\maketitle

\section{Introduction} 

The fractional calculus, during the last years went through a great evolution \cite{RHM,AHMJ,SAM}. The reason for the advancement and expansion is due to the great number of functions that fractional differentiation operators and integration provided to the academic community \cite{ray,oldham,magin,mainardi}. One of the most well-known and important population growth models was introduced by Thomas Malthus in 1798, an English economist \cite{porco}. The Malthus model predicts that the population rate of a given region at time $t$, is proportional to the total population at that instant. It has been found that non-variable order fractional derivatives more accurately describe population growth \cite{almeida}. There are some others mathematical models related to population growth, we suggest \cite{crow,jumarie,almeida3}.

Solutions of the integral equations, in the Volterra or Fredholm sense, has gained prominence, especially in the fields of science and engineering \cite{IE10,IE11}. Also, the investigation of the theorems associated with existence and uniqueness of solutions of nonlinear integral equations has been shown to be important for several researchers \cite{IE,IE1,IE2,IE3}. On the other hand, to study the existence and stability of solutions of fractional integral equations, through integral operators such that: Riemann-Liouville, Hadamard and Katugampola, has been gaining prominence in both analytical sense and in the sense of functions \cite{IE4}. On the other hand, recent studies on $\beta-$distance functions, proved to be important in the study of the existence of solutions of integral equations and fractional differential equations \cite{wongyat,wongyat1,wongyat2}.

However, with the recent development of the fractional calculus, in particular, the introduction of non-variable integrals and fractional derivatives with respect to another function, among which we mention the fractional derivatives: ${\rm \Psi-RL}$, ${\rm \Psi}$-Caputo $({\rm \Psi-C})$ and ${\rm \Psi-H}$ \cite{RHM,AHMJ,JEM1,RCA}, motivated the mathematicians, in particular Almeida \cite{porco} among others \cite{ray,mainardi,JEM2,JEM3}, to propose new mathematical models, allowing a better description of the respective models. In order to propose a fractional differentiation operator and at the same time to unify a class of fractional operators, Sousa and Oliveira \cite{JEM1} introduced the ${\rm \Psi-H}$ fractional derivative. 

\textbf{Since the ${\rm \Psi-H}$ fractional derivative holds a large class of fractional derivatives and also the integer fractional operator in this paper, we have a greater motivation to use the fractional integration operators in relation to another function and the ${\rm \Psi-RL}$ and the ${\rm \Psi-H}$ fractional derivatives. First, it has been observed that the definition of ${\rm \Psi}-$fractional integrals different from other definitions of fractional integrals, since the ${\rm \Psi-RL}$ fractional derivative is used, in addition to contain some particular case \cite{xu}. Second, that from some interesting new results will be investigated, besides proposing applications involving ${\rm \Psi}-$fractional VIE.}

In this sense, considering the nonlinear ${\rm \Psi}$-fractional ${\rm VIE}$ by means of the operator $\mathbf{I}^{\mu,\nu;\rm\Psi}_{a+}(\cdot)$ Eq.(\ref{jose}) given by
\begin{equation}\label{cont98}
x\left( t\right) =\varphi \left( t\right) +\mathbf{I}_{a+}^{\mu ,\nu ;\rm\Psi } \left[ \mathfrak{W}\left( t,s,x\left( s\right) \right) \right] 
\end{equation}
where $0<\mu <1,$ $0\leq \nu \leq 1$, $a,b\in \mathbb{R}$ with $a<b,$ $\phi :\Delta \rightarrow \mathbb{R}$ and $\mathfrak{W}:\Delta ^{2}\times \mathbb{R}\rightarrow \mathbb{R}$ $(\Delta:=\left[ a,b\right])$ are two given function, we will study the uniqueness of solutions through $\beta-$distance functions and $\beta-$generalized $SS$-contraction function.

\textbf{This paper is organized as follows: In section 2, we first introduce the spaces of the continuous functions and functions weighted with their respective norms. In this sense, we present the definition of fractional integral with respect to another function and the ${\rm \Psi-RL}$ and the ${\rm \Psi-H}$ fractional derivative. We present some theorems and lemmas, however we will omit their respective statements. In section 3, we present the main result of the paper. Through the ${\rm \Psi-RL}$ fractional integral and the fractional derivative, we obtain the first principal result, the $\rm\Psi-$fractional integral, and its limitation in the space of weighted functions. In this sense, through the $\rm\Psi-$fractional integral, we present some theorems and important observations to justify the originality of the $\rm\Psi-$fractional integral.} In section 4, we present some examples and graphs: the first example addresses both the $\rm\Psi-$fractional integration operator, the calculation involving the function $ (\rm\Psi(t)-\rm\Psi(a))^{\delta-1}$. On the other hand, the fundamental example in the solution of fractional differential equations involves one parameter ${\rm M-L}$ function \textbf{and also, the calculation of $\rm\Psi-$fractional integral of the ${\rm M-L}$ function of one parameter.} In section 5, we study the uniqueness of solutions of the nonlinear ${\rm\Psi}$-fractional ${\rm VIE}$ by means of the function $\beta-$distance and $\beta-$generalized $\mathbf{SS}$-contraction function. Also, using the ${\rm \Psi-H}$ fractional derivative, we present and solve a fractional version of the Malthus model. Concluding remarks, close the paper.


\section{Preliminaries}

First, we introduce the spaces of the continuous functions and weighted functions, as well as their respective norms. We also present definition involving $\beta-$distance function and $\beta-$generalized $\mathbf{SS}$-contraction function. In this sense, we present ${\rm \Psi-H}$ fractional differentiation operator and fractional integration operator, both with respect to another function. Finally, we present some important theorems and lemmas for the development of the paper.

Let $\Delta=[a,b]$ $(0<a<b<\infty)$ be a finite interval on the half-axis $\mathbb{R}^{+}$ and $C\left(\Delta,\mathbb{R}\right)$ be the space of continuous functions on $\Delta$.

The space of continuous function $f$ on $\Delta$ with the norm is defined by 
\cite{AHMJ} 
\begin{equation*}
\left\Vert f\right\Vert _{C\left(\Delta,\mathbb{R}\right) }=\underset{x\in \Delta }{\max }\left\vert f\left( x\right) \right\vert.
\end{equation*}

The weighted spaces $C_{\xi,\rm\Psi}\left(\Delta,\mathbb{R}\right)$ and $C_{1- \xi,\rm\Psi}\left(\Delta,\mathbb{R}\right)$ of functions $f$ on $I$ are defined by 
\begin{equation*}
C_{\xi ;\rm\Psi }\left(\Delta,\mathbb{R}\right) =\left\{ f:\left( a,b\right] \rightarrow \mathbb{R};\text{ }\left( \rm\Psi (x) -\rm\Psi \left( a\right) \right) ^{\xi }f\left( x\right) \in C\left(\Delta,\mathbb{R}\right) \right\} ,\text{ }0\leq
\xi <1,
\end{equation*}
\begin{equation*}
C_{1-\xi ;\rm\Psi }\left(\Delta,\mathbb{R}\right) =\left\{ f:\left( a,b\right] \rightarrow \mathbb{R};\text{ }\left( \rm\Psi \left( x\right) -\rm\Psi \left( a\right) \right) ^{1-\xi }f\left( x\right) \in C\left(\Delta,\mathbb{R}\right)\right\} ,\text{ }0\leq \xi <1,
\end{equation*}
respectively, with the norm 
\begin{equation}  \label{space}
\left\Vert f\right\Vert _{C_{\xi ;\rm\Psi }\left(\Delta,\mathbb{R}\right)}=\left\Vert \left( \rm\Psi \left( x\right) -\rm\Psi \left( a\right) \right) ^{\xi}f\left( x\right) \right\Vert _{C\left(\Delta,\mathbb{R}\right)}=\underset{x\in \Delta }{\max }\left\vert \left( \rm\Psi \left( x\right) -\rm\Psi \left( a\right) \right) ^{\xi }f\left( x\right) \right\vert
\end{equation}
and 
\begin{equation}  \label{space1}
\left\Vert f\right\Vert _{C_{1-\xi ;\rm\Psi }\left(\Delta,\mathbb{R}\right) }=\left\Vert \left( \rm\Psi \left( x\right) -\rm\Psi \left( a\right) \right) ^{1-\xi}f\left( x\right) \right\Vert _{C\left(\Delta,\mathbb{R}\right)}=\underset{x\in \Delta  }{\max }\left\vert \left( \rm\Psi \left( x\right) -\rm\Psi \left( a\right) \right) ^{1-\xi }f\left( x\right) \right\vert,
\end{equation}
respectively.

The following Definition \ref{def21} and Definition \ref{def22} involve $\beta-$ distance and $\beta-$ generalized $\mathbf{SS}$-contraction, as well as Theorem \ref{teo22}, Theorem \ref{teo23} and Theorem \ref{teo24}, are of paramount importance in the study of the uniqueness of solutions for the nonlinear ${\rm \Psi}$-fractional ${\rm VIE}$, as studied in Section 5 (functions).

\begin{definition} \cite{wongyat,wongyat1} \label{def21} Consider the metric space $\left( \Lambda,\mathbf{d}\right)$. A function $\mathbf{d}:\Lambda\times \Lambda\rightarrow \Delta_{1}$ $(\Delta_{1}:=\left[ 0,\infty \right))$ is called a $\beta -$distance on $\Lambda$ if it satisfies the three conditions $\forall$ $x,y,z\in \Lambda$:
\begin{enumerate}
\item $\mathbf{d}\left( x,y\right) \leq \mathbf{d}\left( x,z\right) +\mathbf{d}\left( z,y\right)$;
\item $\mathbf{d}\left( x,\cdot \right) :\Lambda\rightarrow \Delta_{1} $ is lower semi continuous $\forall$ $x\in \Lambda$;
\item for each $\varepsilon >0$, $\exists\delta >0$ such that $\mathbf{d}\left( x,y\right) \leq \delta $ and $\mathbf{d}\left( x,z\right) \leq \delta \rightarrow \mathbf{d}\left( x,y\right) \leq \varepsilon$.
\end{enumerate}
\end{definition}

\begin{definition}\cite{wongyat}\label{def22} Let $\mathbf{d}$ be a $\beta -$distance on a metric space $\left( \Lambda,\mathbf{d}\right) $. A function $f:\Lambda\rightarrow \Lambda$ is said to be a $\beta -$generalized $\mathbf{SS}$-contraction function if
\begin{equation}\label{eq23}
\mathbf{h}\left( \mathbf{d}\left( fx,fy\right) \right) \leq \phi \left( \mathbf{d}\left( x,y\right) \right) 
\end{equation}
$\forall$ $x,y\in \Lambda$, where $\mathbf{h}:\Delta_{1}\rightarrow \Delta_{1} $ is a weak altering distance function and $\phi :\Delta_{1} \rightarrow \Delta_{1} $ is a right upper semi continuous function such that $\mathbf{h}\left( t\right) >\phi \left( t\right) $ $\forall$ $t>0$.
\end{definition}

\begin{theorem}{\rm \cite{wongyat}}\label{teo22} Let $\left( \Lambda,\mathbf{d}\right) $ be a complete metric space and $\mathbf{d}:\Lambda\times \Lambda\rightarrow \Delta_{1} $ be a $\beta -$distance of $\mathbf{d}$. Suppose that $f:\Lambda\rightarrow \Lambda$ is continuous function such that $\forall$ $x,y\in \Lambda$ 
\begin{equation}\label{eq24}
\mathbf{h}\left( \mathbf{d}\left( fx,fy\right) \right) \leq \mathbf{h} \left( \mathbf{d}\left( x,y\right) \right) -\phi \left( \mathbf{d}\left( x,y\right)\right) 
\end{equation}
where $\mathbf{h}:\Delta_{1} \rightarrow \Delta_{1} $ is an altering distance function and $\phi :\Delta_{1} \rightarrow \Delta_{1} $ is a continuous function with $\phi \left( t\right) =0$ if and only if $t=0.$ Then $f$ has a unique fixed point fixed point in $\Lambda$. Moreover, for each $x_{0}\in \Lambda$, the Picard iteration $\left\{ x_{n}\right\} $, which is defined by $x_{n}=f^{n}x_{0}$ $\forall$ $n\in \mathbb{N}$, converges to a unique fixed point of $f$.
\end{theorem}

\begin{proof}
See \cite{wongyat}.
\end{proof}

\begin{theorem}\label{teo23} Let $\left( \Lambda,\mathbf{d}\right) $ be a complete metric space, $\mathbf{h}\in \Omega $, $\mathbf{d}:\Lambda\times \Lambda\rightarrow \Delta_{1} $ be a $\beta -$distance on $X$ and altering distance on $\mathbf{d}$, and $f:\Lambda\rightarrow \Lambda$ be a function. Suppose that there exist $k_{1}\in \left[ 0,1\right) $ such that
\begin{equation}\label{eq25}
\mathbf{h}\left( \mathbf{d}\left( fx,fy\right) \right) \leq \left[ \mathbf{h}\left( \mathbf{d}\left( x,y\right) \right) \right] ^{k_{1}}
\end{equation}
$\forall$ $x,y\in X$. If you satisfy a condition 
\begin{enumerate}
\item $f:X\rightarrow X$ is a continuous function;
\item $\mathbf{d}$ and $\mathbf{h}$ are continuous functions, then $f$ has a unique fixed point.
\end{enumerate}

Moreover, for each $x_{0}\in \Lambda$, the Picard iteration $\left\{ x_{n}\right\} $, given by $x_{n}=f^{n}x_{0}$ $\forall$ $n\in \mathbb{N} $, converge to a unique fixed point of $f$.
\end{theorem}

\begin{proof}
See \cite{wongyat}.
\end{proof}

\begin{theorem}\label{teo24} Consider the complete metric space $\left( \Lambda,\mathbf{d}\right)$ and $q:\Lambda\times \Lambda\rightarrow \Delta_{1} $ be a $\beta-$distance on $\Lambda$ a distance of $\mathbf{d}$. Suppose $f:\Lambda\rightarrow \Lambda$ is a $\beta -$generalized $\mathbf{SS}$-contraction function such that $\mathbf{h}\left( 0\right) =0.$ Then $f$ has a unique fixed point on $\Lambda$. Moreover, for each $x_{0}\in \Lambda$, the Picard iteration $\left\{x_{n}\right\} $, which is defined by $x_{n}=f^{n}x_{0}$ $\forall$ $n\in \mathbb{N}$, converge to a unique fixed point of $f$.
\end{theorem}

\begin{proof}
See \cite{wongyat}.
\end{proof}

\begin{definition}\cite{AHMJ,SAM} Let $\mu>0$, with $\Delta=[a,b]$ $(-\infty\leq a<x<b\leq \infty)$ a finite or infinite interval, $f$ an integrable function defined on $\Delta$ and $\rm\Psi\in C^{1}(\Delta,\mathbb{R})$ an increasing function such that $\rm\Psi^{\prime }(x)\neq 0$, $\forall x\in \Delta$. The fractional integrals of a function $f$ are given by 
\begin{equation}  \label{P1}
\mathcal{I}_{a+}^{\mu ;\rm\Psi }f\left( x\right) :=\frac{1}{\Gamma \left( \mu \right) }\int_{a}^{x}\mathbf{M}^{\mu}_{\rm\Psi}(x,t)f\left( t\right) dt
\end{equation}
and 
\begin{equation}  \label{P2}
\mathcal{I}_{b-}^{\mu ;\rm\Psi }f\left( x\right) :=\frac{1}{\Gamma \left( \mu \right) }\int_{x}^{b}\mathbf{M}^{\mu}_{\rm\Psi}(x,t)f\left( t\right) dt,
\end{equation}
\end{definition}
with $\mathbf{M}^{\mu}_{\rm\Psi}(x,t):=\rm\Psi ^{\prime }\left( t\right) \left( \rm\Psi \left(x\right) -\rm\Psi \left( t\right) \right) ^{\mu -1}$.

applying $\mu=1$ in Eq.(\ref{P1}) and Eq.(\ref{P2}) , we have 
\begin{equation}  \label{P3}
I_{a+}^{1 ;\rm\Psi }f\left( x\right)=\int_{a}^{x}\rm\Psi ^{\prime }\left( t\right) f\left( t\right) dt
\end{equation}
and 
\begin{equation}  \label{P4}
I_{b-}^{1 ;\rm\Psi }f\left(x\right)=\int_{x}^{b}\rm\Psi ^{\prime }\left( t\right) f\left( t\right) dt.
\end{equation}

\textbf{In the course of the paper we will use the ${\rm \Psi(\cdot)}$ function a lot. For a better development, everywhere that such a function is found, subtends already that it is increasing.}

In this sense, from Eq.(\ref{P1}) and Eq.(\ref{P2}), the Riemann-Liouville fractional derivatives of $f$ of order $\mu $, are defined by \cite{AHMJ,SAM} 
\begin{eqnarray}
\mathcal{D}_{a+}^{\mu ;\rm\Psi }f\left( x\right) &=&\left( \frac{1}{\rm\Psi
^{\prime }\left( x\right) }\frac{d}{dx}\right) \mathcal{I}_{a+}^{1-\mu ;\rm\Psi
}f\left( x\right)  \notag  \label{D2} \\
&=&\frac{1}{\Gamma \left( 1-\mu \right) }\left( \frac{1}{\rm\Psi ^{\prime
}\left( x\right) }\frac{d}{dx}\right) \int_{a}^{x}\mathbf{Q}^{\mu}_{\rm\Psi}(x,t) f\left( t\right) dt  \notag \\
&&
\end{eqnarray}%
and 
\begin{eqnarray}
\mathcal{D}_{b-}^{\mu ;\rm\Psi }f\left( x\right) &=&\left( -\frac{1}{\rm\Psi
^{\prime }\left( x\right) }\frac{d}{dx}\right) \mathcal{I}_{a+}^{1-\mu ;\rm\Psi
}f\left( x\right)  \notag  \label{D3} \\
&=&\frac{1}{\Gamma \left( n-\mu \right) }\left( -\frac{1}{\rm\Psi ^{\prime
}\left( x\right) }\frac{d}{dx}\right) \int_{x}^{b} \mathbf{Q}^{\mu}_{\rm\Psi}(x,t)  f\left( t\right) dt,  \notag \\
&&
\end{eqnarray}
with $\mathbf{Q}^{\mu}_{\rm\Psi}(x,t):= \rm\Psi ^{\prime }\left(
t\right) \left( \rm\Psi \left( t\right) -\rm\Psi \left( x\right) \right) ^{-\mu
}$.

\begin{remark}
From {\rm Eq.(\ref{P1})} and {\rm Eq.(\ref{P2})}, we also consider a fractional derivative version, the so-called ${\rm \Psi-C}$ fractional derivative, recently introduced by Almeida {\rm \cite{RCA}}.
\end{remark}

\begin{definition}{\rm \cite{JEM1}} Let $0<\mu <1$ and $f,\rm\Psi \in C^{1}(\Delta,\mathbb{R})$ such that $\rm\Psi ^{\prime }(x)\neq 0$, $\forall$ $x\in \Delta$. The ${\rm \Psi-H}$ fractional derivatives $(\mbox{left-sided and right-sided})$ $^{H}\mathscr{D}_{a+}^{\mu ,\nu ;\rm\Psi }(\cdot )$ and $^{H}\mathscr{D}_{b-}^{\mu ,\nu ;\rm\Psi }(\cdot )$ of a function of order $\mu $ and type $0\leq \nu \leq 1$, are defined by 
\begin{equation}
^{H}\mathscr{D}_{a+}^{\mu ,\nu ;\rm\Psi }f\left( x\right) =\mathcal{I}_{a+}^{\nu
\left( 1-\mu \right) ;\rm\Psi }\left( \frac{1}{\rm\Psi ^{\prime }\left(
x\right) }\frac{d}{dx}\right) \mathcal{I}_{a+}^{\left( 1-\nu \right) \left( 1-\mu
\right) ;\rm\Psi }f\left( x\right)  \label{HIL}
\end{equation}%
and 
\begin{equation}
^{H}\mathscr{D}_{b-}^{\mu ,\nu ;\rm\Psi }f\left( x\right) =\mathcal{I}_{b-}^{\nu
\left( 1-\mu \right) ;\rm\Psi }\left( -\frac{1}{\rm\Psi ^{\prime }\left(
x\right) }\frac{d}{dx}\right) \mathcal{I}_{b-}^{\left( 1-\nu \right) \left( 1-\mu
\right) ;\rm\Psi }f\left( x\right),  \label{HIL1}
\end{equation}
respectively.

The Eq.(\ref{HIL}) and Eq.(\ref{HIL1}), can be written as
\begin{equation}
^{H}\mathscr{D}_{a+}^{\mu ,\nu ;\rm\Psi }f\left( x\right) =\mathcal{I}_{a+}^{\xi
-\mu ;\rm\Psi }\mathcal{D}_{a+}^{\xi ;\rm\Psi }f\left( x\right)  \label{HIL2}
\end{equation}%
and 
\begin{equation}
^{H}\mathscr{D}_{b-}^{\mu ,\nu ;\rm\Psi }f\left( x\right) =\mathcal{I}_{b-}^{\xi
-\mu ;\rm\Psi }\left( -1\right) \mathcal{D}_{b-}^{\xi ;\rm\Psi }f\left(
x\right) ,  \label{HIL3}
\end{equation}%
respectively, with $\xi =\mu +\nu \left( 1-\mu \right) $ and $\mathcal{I}_{a+}^{\xi -\mu ;\rm\Psi }(\cdot )$, $\mathcal{D}_{a+}^{\xi ;\rm\Psi }(\cdot )$, $\mathcal{I}_{b-}^{\xi -\mu ;\rm\Psi }(\cdot )$, $\mathcal{D}_{b-}^{\xi ;\rm\Psi }(\cdot )$ as defined in {\rm Eq.(\ref{P1})}, {\rm Eq.(\ref{D2})}, {\rm Eq.(\ref%
{P2})} and {\rm Eq.(\ref{D3})}.
\end{definition}

\begin{lemma} \label{LE} Let $\mu>0$ and $\nu>0$. Then, we have 
\begin{equation*}
\mathcal{I}_{a+}^{\mu ;\rm\Psi }\mathcal{I}_{a+}^{\nu ;\rm\Psi }f\left( x\right) =\mathcal{I}_{a+}^{\mu +\nu ;\rm\Psi }f\left( x\right).
\end{equation*}
\end{lemma}

\begin{proof}
See {\rm \cite{AHMJ}}.
\end{proof}

\begin{lemma} \label{lema2} Let $\mu>0$ and $\delta>0$. Consider the function $f(x)= \left( \rm\Psi \left( x\right) -\rm\Psi \left( a\right) \right) ^{\delta -1}$, then 
\begin{equation*}
\mathcal{I}_{a+}^{\mu ;\rm\Psi }f(x)=\frac{\Gamma \left( \delta \right) }{\Gamma \left(
\mu +\delta \right) }\left( \rm\Psi \left( x\right) -\rm\Psi \left( a\right)
\right) ^{\mu +\delta -1}.
\end{equation*}
\end{lemma}

\begin{proof}
See {\rm \cite{AHMJ}}.
\end{proof}

\begin{lemma}\label{lema3} Given $\delta\in\mathbb{R}$, consider the function
$f\left( x\right) =\left( \rm\Psi \left( x\right) -\rm\Psi \left( a\right) \right) ^{\delta -1}$, where $\delta > 0$. Then, for $0<\mu<1$ and $0 \leq \nu\leq 1$,
\begin{equation*}
\text{ }^{H}\mathscr{D}_{a+}^{\mu ,\nu ;\rm\Psi }f\left( x\right) =\frac{\Gamma \left(
\delta \right) }{\Gamma \left( \delta -\mu \right) }\left( \rm\Psi \left(
x\right) -\rm\Psi \left( a\right) \right) ^{\delta -\mu -1}.
\end{equation*}
\end{lemma}
\begin{proof}
See {\rm \cite{JEM1}}.
\end{proof}

\begin{theorem}\label{teo1} If $f\in C^{1}(\Delta, \mathbb{R})$, $0<\mu<1$ and $0\leq \nu \leq 1$, then
\begin{equation}\label{K2}
\mathcal{I}_{a+}^{\mu ;\rm\Psi }\text{ }^{H}\mathscr{D}_{a+}^{\mu ,\nu ;\rm\Psi
}f\left( x\right) =f\left( x\right) -\Theta^{\xi}(x,a) \mathcal{I}_{a+}^{1-\xi ;\rm\Psi }f\left( a\right) ,
\end{equation}
with $\xi=\mu+\nu(1-\mu)$ and $\Theta^{\xi}(x,a):=\dfrac{\left( \rm\Psi \left( x\right) -\rm\Psi
\left( a\right) \right) ^{\xi -1}}{\Gamma \left( \xi \right) }$.
\end{theorem}
\begin{proof}
See {\rm \cite{JEM1}}.
\end{proof}

\begin{theorem}\label{teorema2} Let $f\in C^{1}(\Delta,\mathbb{R})$, $0<\mu<1$ and $0\leq\nu\leq 1$, we have
\begin{equation*}
^{H}\mathscr{D}_{a+}^{\mu ,\nu ;\rm\Psi }\mathcal{I}_{a+}^{\mu ;\rm\Psi }f\left( x\right)
=f\left( x\right).
\end{equation*}
\end{theorem}

\begin{proof}
See {\rm \cite{JEM1}}.
\end{proof}

\begin{theorem} The ${\rm \Psi-H}$ fractional derivative are bounded operators $\forall$ $0<\mu<1$ and $0 \leq\nu \leq 1$, given by
\begin{equation}\label{L1}
\left\Vert ^{H}\mathscr{D}_{a+}^{\mu ,\nu ;\rm\Psi }f\right\Vert _{C_{\xi ;\rm\Psi}\left(\Delta,\mathbb{R}\right)}\leq K\left\Vert f\right\Vert _{C_{\xi ;\rm\Psi}^{1}\left(\Delta,\mathbb{R}\right)}
\end{equation}
with $K=\dfrac{\left( \rm\Psi \left( b\right) -\rm\Psi \left( a\right) \right)
^{1-\mu }}{\Gamma \left( 2-\xi \right) \Gamma \left( \xi -\mu+1 \right) }$.
\end{theorem}

\begin{proof}
See {\rm \cite{JEM1}}.
\end{proof}


\section{New fractional operator and their properties}

Using the Eq.(\ref{P3}) and the ${\rm \Psi-RL}$ fractional derivative, we define a new operator of fractional integration so called $\rm\Psi-$fractional integral. In this sense, we have considered some observations related to the new fractional integral and ${\rm \Psi-H}$, some of which related to integral and fractional derivative recently introduced by Xu and Agrawal \cite{xu}. \textbf{We present conditions for obtaining the integral operator and the identity operator across the bounds of $\mu\rightarrow 0 $ and $\nu\rightarrow 1 $. We investigate some basic properties for the $\rm\Psi-$fractional integral, called: linearity, limitation, calculations of $\mathbf{I}^{\mu,\nu;\rm\Psi}_{a+}\text{ } {^{H}\mathscr{D}_{a+}^{\mu ,\nu ;\rm\Psi }} (\cdot)$ and ${^{H}\mathscr{D}_{a+}^{\mu ,\nu ;\rm\Psi }} \text{ }\mathbf{I}^{\mu,\nu;\rm\Psi}_{a+}(\cdot)$, as well as other interesting results obtained through the integral $\rm\Psi-$fractional integral.}

First, we present a new type of $\rm\Psi-$fractional integral operator
\begin{definition}
Let $0<\mu <1$, $0\leq\nu \leq 1$, $f\in C\left(\Delta,\mathbb{R}\right)$ and $\rm\Psi \in C^{1}(\Delta,\mathbb{R})$ such that $\rm\Psi ^{\prime }(x)\neq 0$, $\forall x\in \Delta$. The $\rm\Psi-$fractional integrals $(\mbox{left-sided and right-sided})$ $\mathbf{I}_{a+}^{\mu ,\nu ;\rm\Psi}(\cdot )$ and $\mathbf{I}_{b-}^{\mu ,\nu ;\rm\Psi }(\cdot )$ of a function $f$ of order $\mu $ are given by 
\begin{equation}
\mathbf{I}_{a+}^{\mu ,\nu ;\rm\Psi }f\left( x\right) :=\mathcal{D}_{a+}^{\left( 1-\nu \right) \left( 1-\mu \right) ;\rm\Psi }I_{a+}^{1;\rm\Psi }\mathcal{D}_{a+}^{\nu \left( 1-\mu \right) ;\rm\Psi }f\left( x\right) 
\label{jose}
\end{equation}%
and 
\begin{equation}
\mathbf{I}_{b-}^{\mu ,\nu ;\rm\Psi }f\left( x\right) :=\mathcal{D}_{b-}^{\nu \left( 1-\mu \right) ;\rm\Psi }I_{b-}^{1;\rm\Psi }\mathcal{D}_{b-}^{\left( 1-\nu \right) \left( 1-\mu \right) ;\rm\Psi }f\left(x\right) ,  \label{jose1}
\end{equation}%
where $\mathcal{D}_{a+}^{\mu ,\rm\Psi }(\cdot )$, $I_{a+}^{1,\rm\Psi }(\cdot )$, $\mathcal{D}_{b-}^{\mu ,\rm\Psi }(\cdot )$ and $I_{b-}^{1,\rm\Psi }(\cdot )$ are ${\rm \Psi-RL}$ fractional derivatives and fractional integrals according {\rm Eq.(\ref{D2})}, {\rm Eq.(\ref{P3})}, {\rm Eq.(\ref{D3})} and {\rm Eq.(\ref{P4})}, respectively.
\end{definition}
The Eq.(\ref{jose}) and Eq.(\ref{jose1}), can be written as 
\begin{equation*}
\mathbf{I}_{a+}^{\mu ,\nu ;\rm\Psi }f\left( x\right) :=\mathcal{D}_{a+}^{1-\xi ;\rm\Psi }I_{a+}^{1;\rm\Psi }\mathcal{D}_{a+}^{\xi -\mu ;\rm\Psi }f\left( x\right) 
\end{equation*}%
and 
\begin{equation*}
\mathbf{I}_{b-}^{\mu ,\nu ;\rm\Psi }f\left( x\right) :=\mathcal{D}_{b-}^{\xi -\mu ;\rm\Psi }I_{b-}^{1;\rm\Psi }\mathcal{D}_{b-}^{1-\xi ;\rm\Psi }f\left( x\right) ,
\end{equation*}
with $\xi =\mu +\nu (1-\mu )$ and $0<\xi <1$.

Note that, for $\rm\Psi \left( t\right) =t$ and applying the limit $\mu \rightarrow 1$, on both side of the Eq.(3.1), we have
\begin{equation}
\underset{\mu \rightarrow 1^{-}}{\lim }\mathbf{I}_{a+}^{\mu ,\nu ;\rm\Psi }f\left( x\right) =\underset{\mu \rightarrow 1^{-}}{\lim }\mathcal{D}_{a+}^{1-\xi ;\rm\Psi}I_{a+}^{1;\rm\Psi }\mathcal{D}_{a+}^{\xi -\mu ;\rm\Psi }f\left( x\right) =I_{a+}f\left(x\right).
\end{equation}

\textbf{In relation to the integral $\mathcal{I}_{a+}^{\mu }\left( \cdot \right) $, just take}
\begin{equation}
\underset{\mu \rightarrow 0}{\lim }\mathcal{I}_{a+}^{\mu }f\left( x\right)=f\left( x\right) .
\end{equation}

\textbf{Now, applying the limit on both sides of Eq.(19), with $\mu \rightarrow 0$ and $\nu \rightarrow 1$, we obtain}
\begin{equation}
\underset{\nu \rightarrow 1}{\underset{\mu \rightarrow 0}{\lim }}\mathcal{D}_{a+}^{\left( 1-\nu \right) \left( 1-\mu \right) ;\rm\Psi }I_{a+}^{1;\rm\Psi}\mathcal{D}_{a+}^{\nu \left( 1-\mu \right) ;\rm\Psi }f\left( x\right) =f\left(
x\right) .
\end{equation}

\textbf{Note that to obtain the identity operator for Eq.(\ref{jose}), in addition to $\mu \rightarrow 0$, we have to impose the condition $\nu \rightarrow 1$.}

\begin{remark}
applying $\rm\Psi (x)=x$ in {\rm Eq.(\ref{jose})} and {\rm Eq.(\ref{jose1})}, we have integral-type fractional operators as introduced by Xu and Agrawal {\rm \cite{xu}}, given by 
\begin{equation*}
\mathbf{I}_{a+}^{\mu ,\nu ;\rm\Psi}f\left( x\right) =\mathbf{D}_{a+}^{\left(
1-\nu \right) \left( 1-\mu \right) }I_{a+}^{1}\mathbf{D}_{a+}^{\nu
\left( 1-\mu \right) }f\left( x\right) =\ell _{a+}^{\mu ,\nu
}f\left( x\right) 
\end{equation*}%
and 
\begin{equation*}
\mathbf{I}_{b-}^{\mu ,\nu ;\rm\Psi}f\left( x\right) =\mathbf{D}_{b-}^{\nu
\left( 1-\mu \right) }I_{b-}^{1}\mathbf{D}_{b-}^{\left( 1-\nu \right)
\left( 1-\mu \right) }f\left( x\right) =\ell _{b-}^{\mu ,\nu
}f\left( x\right) .
\end{equation*}
\end{remark}

\begin{remark}
Recently, Sousa and Oliveira {\rm \cite{JEM1}} introduced the ${\rm \Psi-H}$ fractional differentiation operator and proved that it is possible to obtain a class of fractional operators, in particular the ${\rm \Psi-RL}$ fractional operator as used to define the above fractional integral. In this sense, from an adequate choice of $\rm\Psi (x) $ we obtain a class of fractional integration operators from {\rm Eq.(\ref{jose})} and {\rm Eq.(\ref{jose1})}.
\end{remark}

\textbf{One of the first properties when proposing a fractional integral is given by the following theorem.}

\begin{theorem} \textbf{Let $0<\mu <1$, $\ f,g\in C\left(\Delta,\mathbb{R}\right) $ and $\rm\Psi \in C^{1}\left(\Delta,\mathbb{R}\right) $ such that $\rm\Psi ^{\prime }\left( x\right)\neq 0$ $\forall x\in \Delta$. Then, we have the following properties:}
\begin{enumerate}

\item \textbf{$\mathbf{I}_{a+}^{\mu ,\nu ;\rm\Psi }\left( \lambda f\pm g\right) \left(x\right) =\lambda \mathbf{I}_{a+}^{\mu ,\nu ;\rm\Psi }f\left( x\right) \pm \mathbf{I}_{a+}^{\mu ,\nu ;\rm\Psi }g\left( x\right) ,$ for $\lambda \in \mathbb{R}$;}

\item \textbf{If $t=a$, then $\mathbf{I}_{a+}^{\mu ,\nu ;\rm\Psi }f\left( a\right) =0$;}

\item \textbf{If $f\left( x\right) \geq 0,$ then $\mathbf{I}_{a+}^{\mu ,\nu ;\rm\Psi }f\left(x\right) \geq 0$.}
\end{enumerate}
\end{theorem}

\begin{proof}
\textbf{By means of the definitions Eq.(\ref{P3}), Eq.(\ref{P4}), Eq.(\ref{D2}) and Eq.(\ref{D3}), the proof is performed.}
\end{proof}

Next, we present a theorem that guarantees that the $\rm\Psi-$fractional integral is bounded, that is, it is well defined.

\begin{theorem} Let $0<\mu<1$, $0\leq \nu \leq 1$ and $\mathcal{D}^{\mu;\rm\Psi}_{a+}(\cdot)$ the ${\rm \Psi-RL}$ fractional derivative. If $f\in C\left(\Delta,\mathbb{R}\right)$, then 
\begin{equation}
\left\Vert \mathbf{I}_{a+}^{\mu ,\nu ;\rm\Psi }f\right\Vert _{C_{\xi ,\rm\Psi
}\left(\Delta,\mathbb{R}\right)}\leq s\left\Vert f\right\Vert _{C_{\xi ,\rm\Psi }^{\left[ 1\right] }\left(\Delta,\mathbb{R}\right)}
\end{equation}
with
\begin{equation*}
s=\frac{\left( \rm\Psi \left( b\right) -\rm\Psi \left( a\right) \right) ^{1+\mu
}}{\Gamma \left( 1+\xi \right) \Gamma \left( 2+\mu -\xi \right) }
\end{equation*}
and $\xi =\mu +\nu \left( 1-\mu \right) $.
\end{theorem}

\begin{proof}
First, by applying the limit $\nu\rightarrow 0$ on both sides of {\rm Eq.(\ref{L1})}, we have
\begin{equation}\label{L2}
\left\Vert \mathcal{D}_{a+}^{\mu ;\rm\Psi }f\right\Vert _{C_{\xi ,\rm\Psi }\left(\Delta,\mathbb{R}\right)}\leq
s_{1}\left\Vert f\right\Vert _{C_{\xi ,\rm\Psi }^{\left[ 1\right] }\left(\Delta,\mathbb{R}\right)},
\end{equation}
with
\begin{equation*}
s_{1}=\frac{\left( \rm\Psi \left( b\right) -\rm\Psi \left( a\right) \right)
^{1-\mu }}{\Gamma \left( 2-\mu \right) }.
\end{equation*}

So, by {\rm Eq.(\ref{L2})}, we get
\begin{eqnarray}\label{L3}
\left\Vert \mathbf{I}_{a+}^{\mu ,\nu ;\rm\Psi }f\right\Vert _{C_{\xi ,\rm\Psi }\left(\Delta,\mathbb{R}\right)}
&=&\left\Vert \mathcal{D}_{a+}^{\left( 1-\nu \right) \left( 1-\mu \right) ;\rm\Psi }%
\text{ }I_{a+}^{1;\rm\Psi }\text{ }\mathcal{D}_{a+}^{\nu \left( 1-\mu \right) ;\rm\Psi
}f\right\Vert _{C_{\xi ,\rm\Psi }\left(\Delta,\mathbb{R}\right)}  \notag \\
&\leq &s_{2}\left\Vert I_{a+}^{1;\rm\Psi }\text{ }\mathcal{D}_{a+}^{\nu \left( 1-\mu
\right) ;\rm\Psi }f\right\Vert _{C_{\xi ,\rm\Psi }\left(\Delta,\mathbb{R}\right)},
\end{eqnarray}
with
\begin{equation*}
s_{2}=\frac{\left( \rm\Psi \left( b\right) -\rm\Psi \left( a\right) \right)
^{\xi }}{\Gamma \left( 1+\xi \right) }
\end{equation*}
and $\xi =\mu +\nu \left( 1-\mu \right) $.

On the other hand, we can write
\begin{eqnarray}\label{L4}
\left\Vert I_{a+}^{1;\rm\Psi }\text{ }\mathcal{D}_{a+}^{\nu \left( 1-\mu \right)
;\rm\Psi }f\right\Vert _{C_{\xi ,\rm\Psi }\left(\Delta,\mathbb{R}\right)} &=&\underset{x\in \Delta 
}{\max }\left\vert \left( \rm\Psi \left( x\right) -\rm\Psi \left( a\right) \right)
^{\xi }I_{a+}^{1;\rm\Psi }\text{ }\mathcal{D}_{a+}^{\nu \left( 1-\mu \right)
;\rm\Psi }f\left( x\right) \right\vert   \notag \\
&\leq &\left\Vert \mathcal{D}_{a+}^{\nu \left( 1-\mu \right) ;\rm\Psi }f\right\Vert
_{C_{\xi ,\rm\Psi }\left(\Delta,\mathbb{R}\right)}\underset{x\in \Delta }{\max }\left\vert
I_{a+}^{1;\rm\Psi }(1)\right\vert   \notag \\
&\leq &\left( \rm\Psi \left( b\right) -\rm\Psi \left( a\right) \right) \left\Vert
\mathcal{D}_{a+}^{\nu \left( 1-\mu \right) ;\rm\Psi }f\right\Vert _{C_{\xi ,\rm\Psi
}}  \notag \\
&\leq &s_{3}\left\Vert f\right\Vert _{C_{\xi ,\rm\Psi }^{\left[ 1\right] }\left(\Delta,\mathbb{R}\right)},
\end{eqnarray}
with
\begin{equation*}
s_{3}=\frac{\left( \rm\Psi \left( b\right) -\rm\Psi \left( a\right) \right) ^{1-\nu \left( 1-\mu \right) }}{\Gamma \left( 2-\nu \left( 1-\mu \right) \right) }.
\end{equation*}

So, using {\rm Eq.(\ref{L3})} and {\rm Eq.(\ref{L4})}, we conclude that
\begin{eqnarray*}
\left\Vert \mathbf{I}_{a+}^{\mu ,\nu ;\rm\Psi }f\right\Vert _{C_{\xi ,\rm\Psi }\left(\Delta,\mathbb{R}\right)}
&\leq &s_{2}\left\Vert I_{a+}^{1;\rm\Psi }\text{ }\mathcal{D}_{a+}^{\nu \left( 1-\mu
\right) ;\rm\Psi }f\right\Vert _{C_{\xi ,\rm\Psi }\left(\Delta,\mathbb{R}\right)}  \notag \\
&\leq &s_{3}s_{2}\left\Vert f\right\Vert _{C_{\xi ,\rm\Psi }^{\left[ 1\right]
}\left(\Delta,\mathbb{R}\right)}  \notag \\
&=&s\left\Vert f\right\Vert _{C_{\xi ,\rm\Psi }^{\left[ 1\right] }\left(\Delta,\mathbb{R}\right)}
\end{eqnarray*}
with
\begin{equation*}
s=\frac{\left( \rm\Psi \left( b\right) -\rm\Psi \left( a\right) \right) ^{1+\mu
}}{\Gamma \left( 1+\xi \right) \Gamma \left( 2+\mu -\xi \right) },
\end{equation*}
which complete the proof.
\end{proof}

\begin{theorem} Let $0<\mu <1$, $0\leq \nu \leq 1$ and $f,\rm\Psi \in C\left(\Delta,\mathbb{R}\right)$ such that $\rm\Psi ^{\prime }\left( t\right) \neq 0$, $\forall$ $t\in \left[ a,b\right) $, a closed interval. Then, we have
\begin{equation}
\left\Vert \mathbf{I}_{a+}^{\mu ,\nu ;\rm\Psi }f-\mathbf{I}_{a+}^{\mu ,\nu ;\rm\Psi }g\right\Vert _{C_{\xi ;\rm\Psi }\left(\Delta,\mathbb{R}\right)}\leq K\left\Vert f-g\right\Vert_{C_{\xi ;\rm\Psi }\left(\Delta,\mathbb{R}\right)}.
\end{equation}
\end{theorem}

\begin{proof} 
In fact, by definition of $\rm\Psi-$fractional integral $\mathbf{I}_{a+}^{\mu ,\nu ;\rm\Psi }\left(\cdot \right) $, we have
\begin{eqnarray}
\left\Vert \mathbf{I}_{a+}^{\mu ,\nu ;\rm\Psi }f-\mathbf{I}_{a+}^{\mu ,\nu ;\rm\Psi }g\right\Vert _{C_{\xi ;\rm\Psi }\left(\Delta,\mathbb{R}\right)} &\leq &\left\Vert f-g\right\Vert _{C_{\xi ;\rm\Psi }\left(\Delta,\mathbb{R}\right)}\underset{x\in \Delta }{\max }\left\vert \mathcal{D}_{a+}^{1-\xi ;\rm\Psi }I_{a+}^{1;\rm\Psi }\mathcal{D}_{a+}^{\xi -\mu ;\rm\Psi}(1)\right\vert   \notag \\ 
&\leq &K\left\Vert f-g\right\Vert _{C_{\xi ;\rm\Psi }\left(\Delta,\mathbb{R}\right)}
\end{eqnarray}
\textbf{with $K>0$.}
\end{proof}

The Theorem \ref{teo310} below involves uniformly convergent sequence of continuous functions being important in the context of analysis within the fractional calculus.

\begin{theorem}\label{teo310} Let $0<\mu <1$ and $0\leq \nu \leq 1$. Suppose that $\left( f_{k}\right) _{k=1}^{\infty }$ is a uniformly convergent sequence of continuous functions on $\Delta $. Then, we have
\begin{equation}
\mathbf{I}_{a+}^{\mu ,\nu ;\rm\Psi }\underset{k\rightarrow \infty }{\lim }f_{k}\left( x\right)= \underset{k\rightarrow \infty }{\lim }\mathbf{I}_{a+}^{\mu,\nu ;\rm\Psi }f_{k}\left( x\right).
\end{equation}

Moreover, in particular, note that, the sequence of function $\left( \mathbf{I}_{a+}^{\mu ,\nu ;\rm\Psi }f_{k}\right) _{k=1}^{\infty }$ is uniformly convergent.
\end{theorem}

\begin{proof} As $\mathcal{D}_{a+}^{\mu ;\rm\Psi }\left( \cdot \right) $ and $ \mathbf{I}_{a+}^{\mu ;\rm\Psi }\left( \cdot \right) $ there are uniformly convergent, we have
\begin{eqnarray*}
&&\left\vert \mathbf{I}_{a+}^{\mu ,\nu ;\rm\Psi }f_{k}\left( x\right) - \mathbf{I}_{a+}^{\mu ,\nu ;\rm\Psi }f\left( x\right) \right\vert   \notag\\
&=&\left\vert \mathcal{D}_{a+}^{1-\xi ;\rm\Psi }I_{a+}^{1;\rm\Psi }\mathcal{D}_{a+}^{\xi -\mu ;\rm\Psi }f_{k}\left( x\right) -\mathcal{D}_{a+}^{1-\xi ;\rm\Psi }I_{a+}^{1;\rm\Psi }\mathcal{D}_{a+}^{\xi -\mu
;\rm\Psi }f\left(x\right) \right\vert \rightarrow 0
\end{eqnarray*}
applying $k\rightarrow \infty$.
\end{proof}

\begin{theorem}\label{teo311} Let $0<\mu<1$ and $0<\nu<1$. If $f$ is differentiable and integrable, then
\begin{equation}
\mathbf{I}_{a+}^{\mu ,\nu ;\rm\Psi }\text{ } ^{H}\mathscr{D}_{a+}^{\mu ,\nu ;\rm\Psi } f\left( x\right) =f\left( x\right) -\Theta^{\xi}(x,a)\mathcal{I}_{a+}^{1-\xi ;\rm\Psi}f\left( a\right) ,
\end{equation}
with $\xi =\mu +\nu \left( 1-\mu \right) $.
\end{theorem}

\begin{proof}
Using {\rm Theorem \ref{teorema2}} with $\nu \rightarrow 0$, we have
\begin{eqnarray}\label{X1}
\mathbf{I}_{a+}^{\mu ,\nu ;\rm\Psi }\text{ }^{H}\mathscr{D}_{a+}^{\mu ,\nu ;\rm\Psi }f\left( x\right) 
&=&\mathcal{D}_{a+}^{\left( 1-\nu \right) \left( 1-\mu \right) ;\rm\Psi }I_{a+}^{1;\rm\Psi }\left( \frac{1}{\rm\Psi ^{\prime }\left( x\right) }\frac{d}{dx} \right) \mathcal{I}_{a+}^{\left( 1-\nu \right) \left( 1-\mu \right) ;\rm\Psi
}f\left( x\right)   \notag \\
&=&\mathcal{D}_{a+}^{1-\xi ;\rm\Psi }\mathcal{I}_{a+}^{1-\xi;\rm\Psi }\mathcal{I}_{a+}^{\xi ;\rm\Psi }\mathcal{D}_{a+}^{\xi ;\rm\Psi }f\left( x\right).
\end{eqnarray}

Now, applying the limit $\nu \rightarrow 0$ on both sides of the {\rm Eq.(\ref{K2})}, we get
\begin{equation}\label{X2}
\mathcal{I}_{a+}^{\mu ;\rm\Psi }\mathcal{D}_{a+}^{\mu ;\rm\Psi }f\left( x\right) =f\left(
x\right) -\Theta^{\mu}(x,a)\mathcal{I}_{a+}^{1-\mu ;\rm\Psi }f\left(
a\right).
\end{equation}

Using {\rm Eq.(\ref{X1})}, {\rm Eq.(\ref{X2})} and {\rm Theorem \ref{teorema2}} with $\nu \rightarrow 0$, we conclude
\begin{eqnarray*}
\mathbf{I}_{a+}^{\mu ,\nu ;\rm\Psi }\text{ }^{H}\mathscr{D}_{a+}^{\mu ,\nu ;\rm\Psi }f\left( x\right) 
&=&\mathcal{D}_{a+}^{1-\xi ;\rm\Psi }\mathcal{I}_{a+}^{1-\xi ;\rm\Psi }\left( f\left( x\right) -\Theta^{\xi}(x,a)\mathcal{I}_{a+}^{1-\xi ;\rm\Psi }f\left( a\right)
\right)   \notag \\
&=&f\left( x\right) -\Theta^{\xi}(x,a)\mathcal{I}_{a+}^{1-\xi ;\rm\Psi
}f\left( a\right)
\end{eqnarray*}
which complete the proof.
\end{proof}

\begin{theorem}\label{teo312} Let $0<\mu<1$ and $0<\nu<1$. If $f$ is differentiable and integrable, then
\begin{eqnarray*}
^{H}\mathscr{D}_{a+}^{\mu ,\nu ;\rm\Psi }\text{ }\mathbf{I}_{a+}^{\mu ,\nu ;\rm\Psi }f\left( x\right)  &=&f\left( x\right) -\frac{\left( \rm\Psi \left( x\right)
-\rm\Psi \left( a\right) \right) ^{\nu \left( 1-\mu \right) -1}}{\Gamma
\left( \nu \left( 1-\mu \right) \right) }\mathcal{D}_{a+}^{\nu \left(
1-\mu \right) -1;\rm\Psi }f\left( a\right)   \notag \\
&&-\frac{\left( \rm\Psi \left( x\right) -\rm\Psi \left( a\right) \right) ^{-\mu
-1}}{\Gamma \left( -\mu \right) }\mathcal{D}_{a+}^{-\mu -1;\rm\Psi
}f\left( a\right) .
\end{eqnarray*}

\end{theorem}
\begin{proof}
In fact, making use of Eq.(\ref{HIL}) and applying the limit $\nu\rightarrow 0$ on both sides of {\rm Eq.(\ref{K2})}, we get
\begin{eqnarray*}
&&^{H}\mathscr{D}_{a+}^{\mu ,\nu ;\rm\Psi } \text{ }\mathbf{I}_{a+}^{\mu ,\nu ;\rm\Psi
}f\left( x\right)  \\
&=&\mathcal{I}_{a+}^{\nu \left( 1-\mu \right) ;\rm\Psi }\left( \frac{1}{\rm\Psi
^{\prime }\left( x\right) }\frac{d}{dx}\right) \mathcal{I}_{a+}^{1-\xi ;\rm\Psi }%
\mathcal{D}_{a+}^{1-\xi ;\rm\Psi }I_{a+}^{1;\rm\Psi }\mathcal{D}_{a+}^{\nu
\left( 1-\mu \right) ;\rm\Psi }f\left( x\right)  \\
&=&\mathcal{I}_{a+}^{\nu \left( 1-\mu \right) ;\rm\Psi }\mathcal{D}_{a+}^{\nu
\left( 1-\mu \right) ;\rm\Psi }f\left( x\right) -\mathcal{I}_{a+}^{\nu \left( 1-\mu \right) ;\rm\Psi }\left( \frac{1}{\rm\Psi
^{\prime }\left( x\right) }\frac{d}{dx}\right) \frac{\left( \rm\Psi \left(
x\right) -\rm\Psi \left( a\right) \right) ^{-\xi }}{\Gamma \left( 1-\xi
\right) }\mathcal{I}_{a+}^{\xi +1;\rm\Psi }\mathcal{D}_{a+}^{\nu \left( 1-\mu
\right) ;\rm\Psi }f\left( a\right)  \\
&=&f\left( x\right) -\frac{\left( \rm\Psi \left( x\right) -\rm\Psi \left( a\right)
\right) ^{\nu \left( 1-\mu \right) -1}}{\Gamma \left( \nu \left(
1-\mu \right) \right) }\mathcal{I}_{a+}^{1-\nu \left( 1-\mu \right) ;\rm\Psi
}f\left( a\right)  \\
&&-\mathcal{I}_{a+}^{\nu \left( 1-\mu \right) ;\rm\Psi }\left( \frac{\left( \rm\Psi
\left( x\right) -\rm\Psi \left( a\right) \right) ^{-\xi -1}}{\Gamma \left(
-\xi \right) }\mathcal{I}_{a+}^{\xi +1;\rm\Psi }\mathcal{D}_{a+}^{\nu \left(
1-\mu \right) ;\rm\Psi }f\left( a\right) \right) \\
&=&f\left( x\right) -\frac{\left( \rm\Psi \left( x\right) -\rm\Psi \left( a\right)
\right) ^{\nu \left( 1-\mu \right) -1}}{\Gamma \left( \nu \left(
1-\mu \right) \right) }\mathcal{I}_{a+}^{1-\nu \left( 1-\mu \right) ;\rm\Psi
}f\left( a\right) -\frac{\left( \rm\Psi \left( x\right) -\rm\Psi \left( a\right) \right) ^{-\mu
-1}}{\Gamma \left( -\mu \right) }\mathcal{I}_{a+}^{\xi +1;\rm\Psi }\mathcal{D}%
_{a+}^{\nu \left( 1-\mu \right) ;\rm\Psi }f\left( a\right)   \notag \\
&=&f\left( x\right) -\frac{\left( \rm\Psi \left( x\right) -\rm\Psi \left( a\right)
\right) ^{\nu \left( 1-\mu \right) -1}}{\Gamma \left( \nu \left(1-\mu \right) \right) }\mathcal{D}_{a+}^{\nu \left( 1-\mu \right)
-1;\rm\Psi }f\left( a\right) -\frac{\left( \rm\Psi \left( x\right) -\rm\Psi \left( a\right) \right) ^{-\mu-1}}{\Gamma \left( -\mu \right) }\mathcal{D}_{a+}^{-\mu -1;\rm\Psi
}f\left( a\right) ,
\end{eqnarray*}
which complete the proof.
\end{proof}

\begin{theorem}\textbf{ Given $0<\mu <1$ and $0\leq \nu \leq 1$. Consider $D^{1;\rm\Psi }\left( \cdot \right) =\left( \dfrac{1}{\rm\Psi ^{\prime }\left( x\right) }\dfrac{d}{dx}\right) \left( \cdot \right) $ and $I_{a+}^{1;\rm\Psi }\left( \cdot \right) $ the derivative and integral of integer order with respect to function $\rm\Psi \left( \cdot \right) $, respectively. Then, we have}
\begin{enumerate}
\item $\mathbf{I}_{a+}^{\mu ,\nu ;\rm\Psi }I_{a+}^{1;\rm\Psi }f\left( x\right) =$ $\mathcal{I}_{a+}^{1+\mu ;\rm\Psi }f\left( x\right)$;

\item $\mathbf{I}_{a+}^{\mu ,\nu ;\rm\Psi }D_{a+}^{1;\rm\Psi }f\left( x\right)=D_{a+}^{1-\mu ;\rm\Psi }f\left( x\right)$.
\end{enumerate}
\end{theorem}

\begin{proof} \textbf{1. In fact, using the subtraction on indexes, we have}
\begin{eqnarray*}
\mathbf{I}_{a+}^{\mu ,\nu ;\rm\Psi }I_{a+}^{1;\rm\Psi }f\left( x\right) 
&=&\mathcal{D}_{a+}^{\left( 1-\nu \right) \left( 1-\mu \right) ;\rm\Psi
}I_{a+}^{1;\rm\Psi }\mathcal{D}_{a+}^{\nu \left( 1-\mu \right) ;\rm\Psi }I_{a+}^{1;\rm\Psi
}f\left( x\right)   \notag \\
&=&\mathcal{D}_{a+}^{\left( 1-\nu \right) \left( 1-\mu \right) ;\rm\Psi
}I_{a+}^{1;\rm\Psi }\mathcal{D}_{a+}^{1-\nu \left( 1-\mu \right) ;\rm\Psi }f\left(
x\right)   \notag \\
&=&\mathcal{I}_{a+}^{2-\nu \left( 1-\mu \right) -\left( 1-\nu \right) \left(
1-\mu \right) ;\rm\Psi }f\left( x\right)   \notag \\
&=&\mathcal{I}_{a+}^{1+\mu ;\rm\Psi }f\left( x\right) .
\end{eqnarray*}

\textbf{2. Again, by subtraction on indexes, we have}
\begin{eqnarray*}
\mathbf{I}_{a+}^{\mu ,\nu ;\rm\Psi }D_{a+}^{1;\rm\Psi }f\left( x\right) 
&=&\mathcal{D}_{a+}^{\left( 1-\nu \right) \left( 1-\mu \right) ;\rm\Psi
}I_{a+}^{1;\rm\Psi }\mathcal{D}_{a+}^{\nu \left( 1-\mu \right) ;\rm\Psi }D_{a+}^{1;\rm\Psi
}f\left( x\right)   \notag \\
&=&\mathcal{D}_{a+}^{\left( 1-\nu \right) \left( 1-\mu \right) ;\rm\Psi
}I_{a+}^{1;\rm\Psi }\mathcal{D}_{a+}^{1+\nu \left( 1-\mu \right) ;\rm\Psi }f\left(
x\right)   \notag \\
&=&\mathcal{I}_{a+}^{\nu \left( 1-\mu \right) -\left( 1-\nu \right) \left(
1-\mu \right) ;\rm\Psi }f\left( x\right)   \notag \\
&=&\mathcal{D}_{a+}^{1-\mu ;\rm\Psi }f\left( x\right) .
\end{eqnarray*}
\end{proof}

\begin{theorem} \textbf{Given $0<\mu <1,$ $0\leq \nu \leq 1$ and $\ f\in C\left(\Delta,\mathbb{R}\right) $ and $\rm\Psi \in C^{1}\left(\Delta,\mathbb{R}\right)$ such that $\rm\Psi ^{\prime }\left( x\right)\neq 0$ $\forall x\in \Delta$. Then, we have}
\begin{equation*}
\mathbf{I}_{a+}^{\mu ,\nu ;\rm\Psi }\left( \rm\Psi \left( x\right) f\left( x\right)
\right) =\rm\Psi \left( t\right) \mathcal{I}_{a+}^{\mu ;\rm\Psi }f\left( x\right)
-\mathcal{I}_{a+}^{\mu ;\rm\Psi }f\left( x\right) .
\end{equation*}
\end{theorem}

\begin{proof} \textbf{First, note that}
\begin{eqnarray}\label{jp}
\mathcal{I}_{a+}^{\mu ;\rm\Psi }\left( \rm\Psi \left( x\right) f\left( x\right) \right) 
&=&\frac{1}{\Gamma \left( \mu \right) }\int_{a}^{x}\mathbf{M}^{\mu}_{\xi}(x,t)\left( \rm\Psi \left( x\right) -\left( \rm\Psi \left( x\right) -\rm\Psi \left(
t\right) \right) \right) f\left( t\right) dt  \notag \\
&=&\rm\Psi \left( t\right) \frac{1}{\Gamma \left( \mu \right) }%
\int_{a}^{x}\mathbf{M}^{\mu}_{\xi}(x,t)f\left(t\right) dt-\frac{1}{\Gamma \left( \mu +1\right) }\int_{a}^{x}\mathbf{M}^{\mu}_{\xi}(x,t)f\left(
t\right) dt  \notag \\
&=&\rm\Psi \left( x\right) \mathcal{I}_{a+}^{\mu ;\rm\Psi }f\left( x\right) -\mu
\mathcal{I}_{a+}^{\mu ;\rm\Psi }f\left( x\right) .
\end{eqnarray}

\textbf{By means of Eq.(\ref{jp}), we have}
\begin{eqnarray*}
&&\mathbf{I}_{a+}^{\mu ,\nu ;\rm\Psi }\left( \rm\Psi \left( x\right) f\left( x\right)
\right)   \notag \\
&=&\mathcal{D}_{a+}^{\left( 1-\nu \right) \left( 1-\mu \right) ;\rm\Psi
}I_{a+}^{1;\rm\Psi }\mathcal{D}_{a+}^{\nu \left( 1-\mu \right) ;\rm\Psi }\left( \rm\Psi
\left( x\right) f\left( x\right) \right)   \notag \\
&=&\mathcal{D}_{a+}^{\left( 1-\nu \right) \left( 1-\mu \right) ;\rm\Psi
}\mathcal{I}_{a+}^{1-\nu \left( 1-\mu \right) ;\rm\Psi }\left( \rm\Psi \left( x\right)
f\left( x\right) \right)   \notag \\
&=&\mathcal{D}_{a+}^{\left( 1-\nu \right) \left( 1-\mu \right) ;\rm\Psi }\left( \rm\Psi
\left( t\right) \mathcal{I}_{a+}^{1-\nu \left( 1-\mu \right) ;\rm\Psi }f\left(
x\right) -\left( 1-\nu \left( 1-\mu \right) \right) \mathcal{I}_{a+}^{1-\nu
\left( 1-\mu \right) ;\rm\Psi }f\left( x\right) \right)   \notag \\
&=&\mathcal{I}_{a+}^{-\left( 1-\nu \right) \left( 1-\mu \right) ;\rm\Psi }\left[
\rm\Psi \left( t\right) \mathcal{I}_{a+}^{1-\nu \left( 1-\mu \right) ;\rm\Psi }f\left(
x\right) \right] -\mathcal{I}_{a+}^{-\left( 1-\nu \right) \left( 1-\mu \right)
;\rm\Psi }\left[ \left( 1-\nu \left( 1-\mu \right) \right) \mathcal{I}_{a+}^{1-\nu
\left( 1-\mu \right) ;\rm\Psi }f\left( x\right) \right]   \notag \\
&=&\rm\Psi \left( t\right) \mathcal{I}_{a+}^{\mu ;\rm\Psi }f\left( v\right) +\left(
\left( 1-\nu \right) \left( 1-\mu \right) \right) \mathcal{I}_{a+}^{\mu ;\rm\Psi
}f\left( x\right) -\left( 1-\nu \left( 1-\mu \right) \right)
\mathcal{I}_{a+}^{\mu ;\rm\Psi }f\left( x\right)   \notag \\
&=&\rm\Psi \left( t\right) \mathcal{I}_{a+}^{\mu ;\rm\Psi }f\left( x\right)
-\mu \mathcal{I}_{a+}^{\mu ;\rm\Psi }f\left(x\right) .
\end{eqnarray*}
\end{proof}


\section{Examples}

\textbf{As the research studies advance, particularly in the fractional calculus, the functions of ${\rm M-L}$ have been shown to be important and efficient in the study of fractional differential equations. As the exponential function is a solution of linear differential equations with constant coefficients, the ${\rm M-L}$ function is a solution of linear fractional differential equations with constant coefficients. And, therefore, in this sense the ${\rm M-L}$ function can be interpreted as a generalization of the exponential function \cite{ML}.}

\textbf{The ${\rm M-L}$ functions of one and two parameters and some of their properties and relations between them, are presented. Since the study and properties related to the ${\rm M-L}$ functions are broad, we suggest with a more detailed reading the book \cite{ML}.}

\textbf{The purpose of this section, is present some examples by using the $\rm\Psi-$fractional integration operator and the ${\rm \Psi-H}$ fractional differentiation operator involving the one-parameter ${\rm M-L}$ function and the function $\left( \rm\Psi \left( x\right) -\rm\Psi \left( a\right) \right)^{\delta -1}$ and we also present some graphs.}

\textbf{Historically, the gamma function, denoted by $\Gamma(z)$ as presented by Euler, was introduced as a generalization of the concept of factorial, and can be done as an improper integral}
\begin{equation}
\Gamma \left( z\right) =\int_{0}^{\infty }e^{-t}t^{z-1}dt, 
\end{equation}
\textbf{$Re\left( z\right) >0$, with $Re(z)\geq 1$. Considering $\mu=n=0,1,2,...$, we have $\Gamma(n+1)=n!$.}

\textbf{The ${\rm M-L}$ function (one parameter), was introduced by Mittag-Leffler in 1903 \cite{ML1} and is given by the series \cite{ML}}
\begin{equation}\label{A1}
\mathbf{E}_{\xi }\left( z\right) =\overset{\infty }{\underset{k=0}{\sum }}\frac{z^{k}}{\Gamma \left( \xi k+1\right) },
\end{equation}
\textbf{with $\xi \in \mathbb{C}$ and $ {Re}\left( \xi \right) >0$.}

\textbf{Two years after Mittag-Leffler introduced the ${\rm M-L}$ function of a parameter, in 1905 Wiman \cite{ML2}, introduced a generalization, the so-called two parameters ${\rm M-L}$ function, given by the series \cite{ML}}
\begin{equation}\label{A2}
\mathbf{E}_{\xi ,\nu }\left( z\right) =\overset{\infty }{\underset{k=0}{\sum }}%
\frac{z^{k}}{\Gamma \left( \xi k+\nu \right) },
\end{equation}
\textbf{with $\xi ,\nu \in \mathbb{C}$, ${Re}\left( \xi \right) >0$ and ${Re}\left( \nu\right) >0$.}

\textbf{In particular for $\nu=1$, we have $\mathbf{E}_{\xi ,1}\left( z\right) =\mathbf{E}_{\xi }\left( z\right) $, i.e; the Eq.(\ref{A1}). On the order hand, applying $\nu=\xi=1$, we have $\mathbf{E}_{1,1}\left(z\right) =e^{z}$.}

\textbf{Because it is a generalization of the exponential function, the trigonometric and hyperbolic sine and cosine functions are particular cases of the ${\rm M-L}$ function of two parameters. In fact, we have some cases \cite{ML}:}
\begin{enumerate}
\item $\mathbf{E}_{2,1}(z^{2})={\rm cosh(z)}$;

\item $\mathbf{E}_{2,2}(z^{2})={\rm \dfrac{senh(z)}{z}} $;

\item $\mathbf{E}_{2,1}(-z^{2})={\rm cos(z)}$;

\item $\mathbf{E}_{2,2}(-z^{2})={\rm \dfrac{sen(z)}{z}} $;
\end{enumerate}

\textbf{By the definition of the ${\rm M-L}$ function of one parameter, we have}
\begin{equation}
\mathbf{E}_{0}\left( z\right) =\overset{\infty }{\underset{k=0}{\sum }}\frac{z^{k}}{\Gamma \left( 0 k+1\right) }=\overset{\infty }{\underset{k=0}{\sum }}z^{k}.
\end{equation}

\textbf{We note that the series $\displaystyle\overset{\infty }{\underset{k=0}{\sum }}z^{k}$ is a geometric series whose sum is $\dfrac{1}{1-z}$, for $\vert z \vert <1$. Then, for $\vert z \vert <1$, we have $\mathbf{E}_{0}\left( z\right)=\dfrac{1}{1-z}$.}

\textbf{In addition to the comments done above on ${\rm M-L}$ of one and two parameters functions, there are numerous important and interesting properties. In this sense, some properties are presented below.}

\textbf{Let $Re(\mu)>0$, $Re(\nu)>0$ and $t\in \mathbb{C}$. Then, the following properties are true \cite{ML}:}
\begin{enumerate}
\item $\mathbf{E}_{\mu,\nu}(z)=z \mathbf{E}_{\mu,\mu+\nu}(v)+\dfrac{1}{\xi(z)}$;

\item $\mathbf{E}_{\mu,\nu}(z)+ \mathbf{E}_{\mu,\nu}(-z)=2 \mathbf{E}_{\mu,\nu}(z^2)$;

\item $\mathbf{E}_{\mu,\nu}(z)= \nu \mathbf{E}_{\mu,\nu+1}(z)+\mu z \dfrac{1}{dz}\mathbf{E}_{\mu,\nu+1}(z^2)$.
\end{enumerate}

\textbf{The first example to be investigated by means of the $\rm\Psi-$fractional integral, involves the function $\left( \rm\Psi \left( x\right) -\rm\Psi \left( a\right) \right) ^{\delta -1}$, in which the behavior of the calculation involving the function presents an interesting behavior in the graph according to the choice of the function $\rm\Psi(\cdot)$ and the parameters $a$, $\nu$ and $\delta$.}

\begin{lemma}\label{lema4} Let $\delta\in\mathbb{R}$ and $\left( \rm\Psi \left( x\right) -\rm\Psi \left( a\right) \right) ^{\delta -1}$, where  $\delta>0$. Then, for $0<\mu<1$ and $0<\nu<1$, we have
\begin{equation*}
\mathbf{I}_{a+}^{\mu ,\nu ;\rm\Psi }\left( \rm\Psi \left( x\right) -\rm\Psi \left(
a\right) \right) ^{\delta -1}=M^{\mu}_{\delta,\nu}\left( \rm\Psi \left( x\right) -\rm\Psi \left( a\right) \right) ^{\delta -\mu +1},
\end{equation*}	
$M^{\mu}_{\delta,\nu}=\dfrac{\Gamma \left( \delta \right) \Gamma
\left( \delta +\nu \left( 1-\mu \right) \right) }{\Gamma \left( \delta
-\nu \left( 1-\mu \right) \right) \Gamma \left( \delta +2\nu \left(
1-\mu \right) +\mu \right) }$.
\end{lemma}

\begin{proof} To facilitate the development of the calculation we consider the parameters $A=(1-\nu)(1-\mu)$ and $B=\nu(1-\mu)$. Using {\rm Lemma \ref{lema3}} with $\nu\rightarrow 0$ and {\rm Lemma \ref{lema2}}, we have 
\begin{eqnarray*}
\mathbf{I}_{a+}^{\mu ,\nu ;\rm\Psi }\left( \rm\Psi \left( x\right) -\rm\Psi
\left( a\right) \right) ^{\delta -1} &=&\mathcal{D}_{a+}^{A;\rm\Psi }I_{a+}^{1;\rm\Psi }\mathcal{D}_{a+}^{B;\rm\Psi
}\left( \rm\Psi \left( x\right) -\rm\Psi \left( a\right) \right) ^{\delta -1} \\
&=&\frac{\Gamma \left( \delta \right) }{\Gamma \left( \delta -B\right) }%
\mathcal{D}_{a+}^{A;\rm\Psi }I_{a+}^{1;\rm\Psi }\left( \rm\Psi \left( x\right) -\rm\Psi
\left( a\right) \right) ^{B+\delta -1} \\
&=&\frac{\Gamma \left( \delta \right) \Gamma \left( \delta +B\right) }{%
\Gamma \left( \delta -B\right) \Gamma \left( \delta +B+1\right) }\mathcal{D}%
_{a+}^{A;\rm\Psi }\left( \rm\Psi \left( x\right) -\rm\Psi \left( a\right) \right)
^{B+\delta } \\
&=&\frac{\Gamma\left( \delta \right) \Gamma \left( \delta +B\right) }{%
\Gamma \left( \delta -B\right) \Gamma \left( B-A+\delta +1\right) }\left(
\rm\Psi \left( x\right) -\rm\Psi \left( a\right) \right) ^{A+B+\delta }.
\end{eqnarray*}

Substituting the values of $A$ and $B$, we conclude the proof.
\end{proof}

\begin{figure}[h!]
\caption{Graph of $\mathbf{I}^{\protect\mu,\protect\nu;\protect\rm\Psi%
}_{0+}f(x)$, for the kernel $\protect\rm\Psi(x)=x$ with $a=0$, $\protect\nu%
=0.5$ and $\protect\delta=1.5$.}
\label{fig:eryt1}\centering 
\includegraphics[width=13cm]{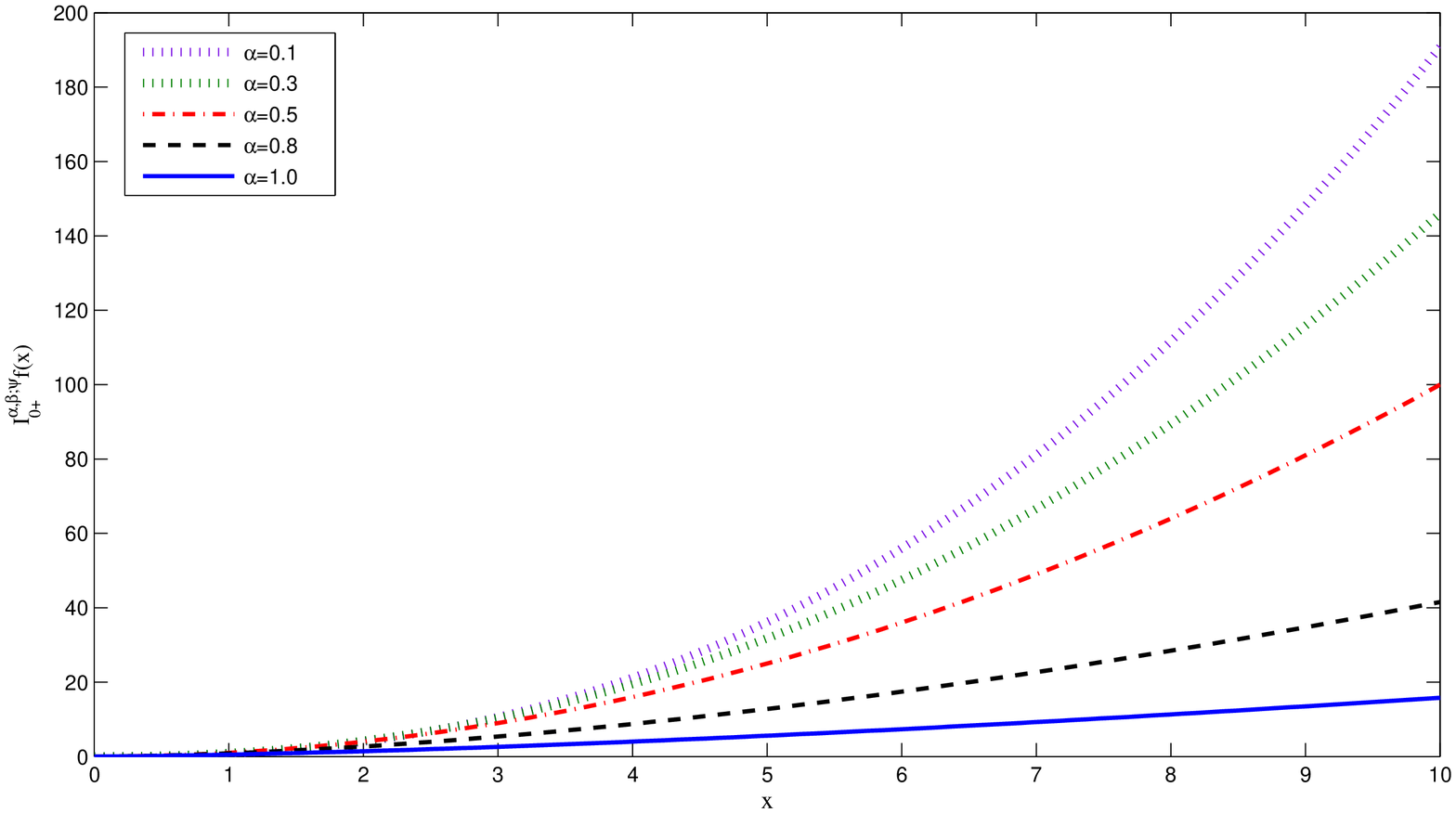}
\end{figure}
\newpage
\begin{figure}[h!]
\caption{Graph of $\mathbf{I}^{\protect\mu,\protect\nu;\protect\rm\Psi%
}_{0+}f(x)$, for the kernel $\protect\rm\Psi(x)=\sqrt[]{x+1}$ with $a=0$, $%
\protect\nu=0.5$ and $\protect\delta=1.5$.}
\label{fig:eryt1}\centering 
\includegraphics[width=13cm]{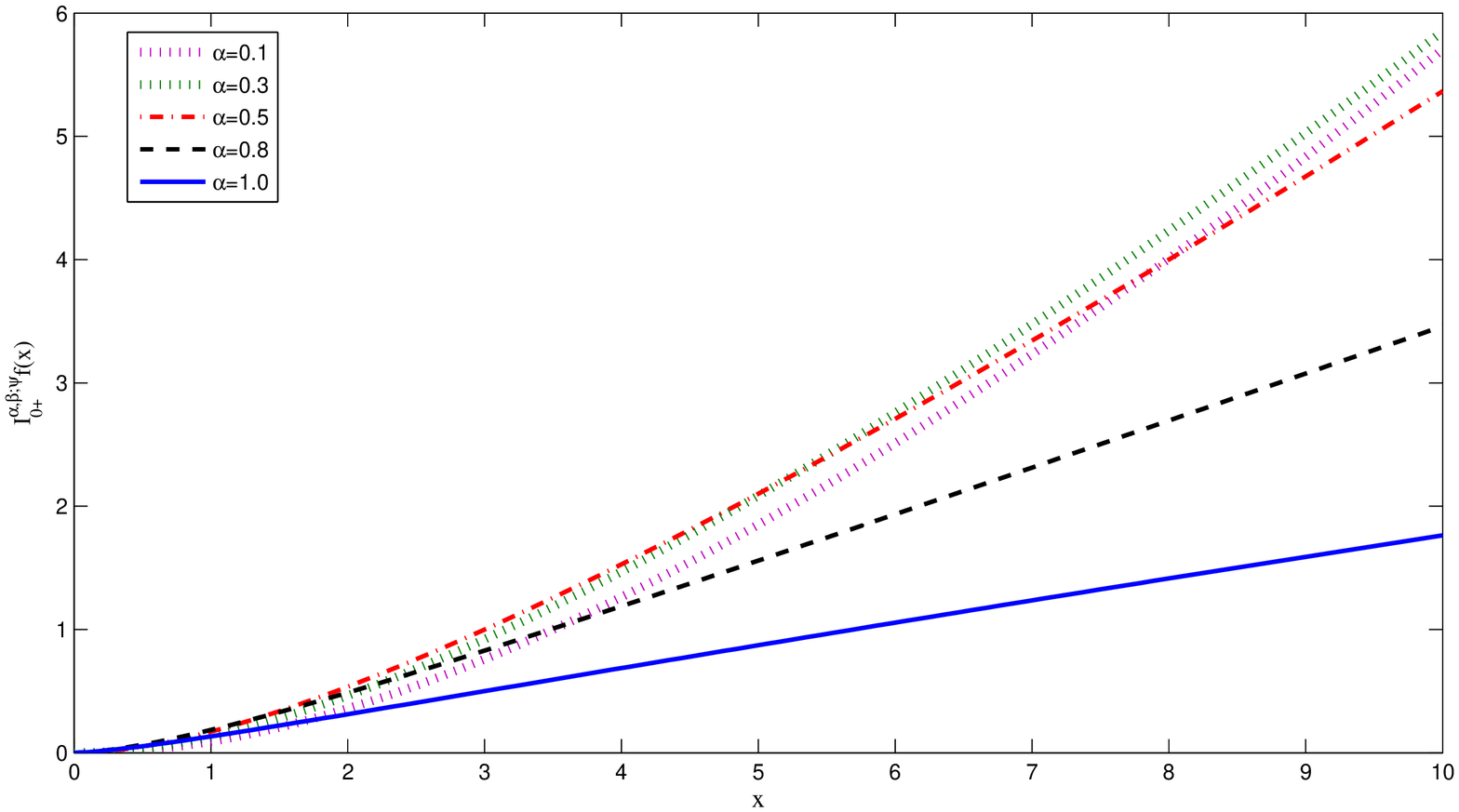}
\end{figure}

\begin{figure}[h!]
\caption{Graph of $\mathbf{I}^{\protect\mu,\protect\nu;\protect\rm\Psi}_{0+}f(x)$, for the kernel $\protect\rm\Psi(x)=\ln x$ with $a=0$, $\protect%
\nu=0.5$ and $\protect\delta=1.5$.}
\label{fig:eryt1}\centering 
\includegraphics[width=13cm]{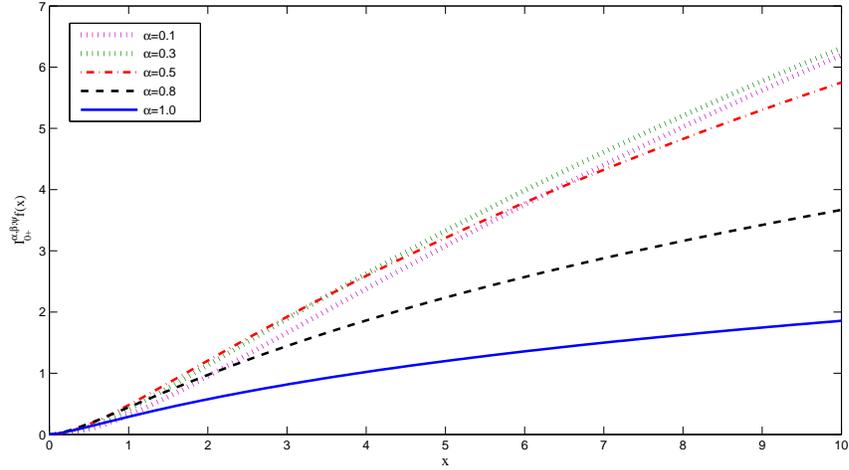}
\end{figure}
\newpage
\textbf{The graphs Fig.1, Fig.2 and Fig.3 are particular cases of Lemma \ref{lema4}. Note that the behavior of the graph Fig.1 is indeed interesting, because as the parameter $\mu$ grows, that is, $\mu = 0.1,0.3,0.5,0.8,1.0 $, the curves perform a behavior contrary. Note that for $\mu = 1.0$, the curve is closer to the $x$ axis than for $\mu = 0.1 $. On the other hand, for Fig. 2 and Fig. 3, their behavior no longer respects a pattern of growth or decay, as seen in the graphs. Such behavior of curves may become important when working with fractional integral equations, in which the calculation of such a function is necessary \cite{ML6}.}

\begin{lemma}\label{lema5} \textbf{Let $\delta\in\mathbb{R}$ and $\left( \rm\Psi \left( x\right) -\rm\Psi \left( a\right) \right) ^{\delta -1}$, where $\delta>0$. Then, for $0<\mu<1$ and $0<\nu<1$, we have}
\begin{equation}
^{H}\mathscr{D}_{a+}^{\mu ,\nu ;\rm\Psi }\left( \rm\Psi \left( x\right) -\rm\Psi \left(
a\right) \right) ^{\delta -1}=\frac{\Gamma \left( \delta \right) }{\Gamma
\left( \delta -\mu \right) }\left( \rm\Psi \left( x\right) -\rm\Psi \left(
a\right) \right) ^{\delta -\mu -1}
\end{equation}
\end{lemma}

\begin{proof} \textbf{See} {\rm \cite{JEM1}}.
\end{proof}

The study of differential equations is of paramount importance in mathematics. The next example will be important for the function that we will perform in the next section.

\begin{lemma}\label{lema6} \textbf{Given $\lambda>0$, $0<\mu<1$ and $0\leq \nu \leq 1$. Let $f\left( x\right) =\mathbf{E}_{\mu }\left[\lambda \left( \rm\Psi \left( x\right) -\rm\Psi \left( a\right) \right) ^{\mu }\right] $, where $\mathbf{E}_{\mu}(\cdot)$ is a one parameter ${\rm M-L}$ function. Then}
\begin{equation*}
^{H}\mathscr{D}_{a+}^{\mu ,\nu ;\rm\Psi }f\left( x\right) =\lambda f\left( x\right) .
\end{equation*}
\end{lemma}

\begin{proof} \textbf{See} {\rm \cite{JEM1}}.
\end{proof}

\begin{lemma} \textbf{Given $\lambda >0,$ $0<\mu <1$ and $0\leq \nu \leq 1$. Let $f\left( x\right) ={\mathbf{E}_{\mu }}\left[ \left( \rm\Psi \left( x\right) -\rm\Psi \left( a\right) \right) ^{\mu }\right] $, where $\mathbf{E}_{\mu }\left( \cdot \right) $ is a one parameter ${\rm M-L}$ function. Then} 
\begin{equation}
\mathbf{I}_{a+}^{\mu ,\nu ;\rm\Psi }f\left( x\right) =\mathbf{E}_{\mu }\left[ \left( \rm\Psi \left( x\right) -\rm\Psi
\left( a\right) \right) ^{\mu }\right] -1.
\end{equation}
\end{lemma}
 
\begin{proof}\textbf{First, note that, by Lemma \ref{lema3}, we have}
\begin{eqnarray}\label{jp1}
\mathcal{D}_{a+}^{\nu \left( 1-\mu \right) ;\rm\Psi }f\left( x\right) 
&=&\mathcal{D}_{a+}^{\nu \left( 1-\mu \right) ;\rm\Psi }\left[ \mathbf{E}_{\mu }\left(
\rm\Psi \left( x\right) -\rm\Psi \left( a\right) \right) ^{\mu }\right]   \notag
\\
&=&\overset{\infty }{\underset{k=0}{\sum }}\frac{1}{\Gamma \left( \mu
k+1\right) }\mathcal{D}_{a+}^{\nu \left( 1-\mu \right) ;\rm\Psi }\left( \rm\Psi \left(
x\right) -\rm\Psi \left( a\right) \right) ^{\mu k} \\
&=&\overset{\infty }{\underset{k=1}{\sum }}\frac{1}{\Gamma \left( \mu
k+1-\nu \left( 1-\mu \right) \right) }\left( \rm\Psi \left( x\right) -\rm\Psi
\left( a\right) \right) ^{\mu k-\nu \left( 1-\mu \right) } \notag.
\end{eqnarray}

\textbf{applying the fractional integral on both sides of the Eq.(\ref{jp1}), we get }
\begin{eqnarray}\label{jp2}
I_{a+}^{1;\rm\Psi }\mathcal{D}_{a+}^{\nu \left( 1-\mu \right) ;\rm\Psi }f\left(
x\right)  &=&\overset{\infty }{\underset{k=1}{\sum }}\frac{1}{\Gamma \left(
\mu k+1-\nu \left( 1-\mu \right) \right) }I_{a+}^{1;\rm\Psi }\left(
\rm\Psi \left( x\right) -\rm\Psi \left( a\right) \right) ^{\mu k-\nu \left(
1-\mu \right) }  \notag \\
&=&\overset{\infty }{\underset{k=1}{\sum }}\frac{1}{\Gamma \left( \mu
k-\nu \left( 1-\mu \right) +2\right) }\left( \rm\Psi \left( x\right) -\rm\Psi
\left( a\right) \right) ^{\mu k-\nu \left( 1-\mu \right) +1}.\notag \\
\end{eqnarray}

\textbf{Finally, applying the fractional derivative $\mathcal{D}_{a+}^{\left( 1-\mu \right) \left( 1-\nu \right) ;\rm\Psi }\left( \cdot \right) $ on both sides of Eq.(\ref{jp2}), we have}
\begin{eqnarray}
\mathbf{I}_{a+}^{\mu ,\nu ;\rm\Psi }f\left( x\right)  &=&\overset{\infty }{\underset
{k=1}{\sum }}\frac{1}{\Gamma \left( \mu k-\nu \left( 1-\mu \right)
+2\right) }\mathcal{D}_{a+}^{\left( 1-\mu \right) \left( 1-\nu \right) ;\rm\Psi
}\left( \rm\Psi \left( x\right) -\rm\Psi \left( a\right) \right) ^{\mu k-\nu
\left( 1-\mu \right) +1}  \notag \\
&=&\overset{\infty }{\underset{k=1}{\sum }}\frac{\left( \rm\Psi \left( x\right)
-\rm\Psi \left( a\right) \right) ^{\mu k-\mu }}{\Gamma \left( \mu
k-\mu +1\right) }-1\notag \\ &=&\mathbf{E}_{\mu }\left[ \left( \rm\Psi \left( x\right) -\rm\Psi
\left( a\right) \right) ^{\mu }\right] -1.
\end{eqnarray}

\end{proof}


\section{Applications}\label{sec5}

In this section we will discuss the uniqueness of the nonlinear $\rm\Psi-$fractional ${\rm VIE}$, by means of $\beta-$distance and $\beta-$generalized $\mathbf{SS}$-contraction function. In addition, we analyze a model described by a differential equation through the ${\rm \Psi-H}$ fractional derivative, and show how the fractional derivatives with respect to another function can be useful in the study of differential equations. The function, is associated with the so-called Malthusian model \cite{porco}, specifically population growth.

First, we present a brief study of the uniqueness of the nonlinear $\rm\Psi-$fractional ${\rm VIE}$. The following results are very important in the studies of $\rm\Psi-$fractional nonlinear integral equations, specially ${\rm VIE}$.

\begin{theorem} Consider the nonlinear $\rm\Psi-$fractional ${\rm VIE}$
\begin{equation}\label{cont51}
x\left( t\right) =\varphi \left( t\right) +\mathbf{I}_{a+}^{\mu ,\nu ;\rm\Psi}\left[ \mathfrak{W}\left( t,s,x\left(s\right) \right) \right] 
\end{equation}
where $0<\mu <1,$ $0\leq \nu \leq 1$, $a,b\in \mathbb{R}$ with $a<b,$ $\varphi :\Delta\rightarrow \mathbb{R}$ and $\mathfrak{W}:\Delta ^{2}\times \mathbb{R}\rightarrow \mathbb{R}$ are two given function. Suppose that $\varphi $, $\mathfrak{W}$ are continuous functions and the function $f:C\left(\Delta,\mathbb{R}\right)\rightarrow C\left(\Delta,\mathbb{R}\right)$ defined by
\begin{equation*}
\left( fx\right) \left( t\right) =\varphi \left( t\right) +\mathbf{I}_{a+}^{\mu ,\nu ;\rm\Psi }\left[ \mathfrak{W}\left( t,s,x\left( s\right) \right) \right] 
\end{equation*}
$\forall$ $x\in C\left(\Delta,\mathbb{R}\right) $ and $t\in \Delta $ is a continuous function. Besides that, let two functions $\mathbf{h},\phi :\Delta_{1} \rightarrow \Delta_{1} $ with $\mathbf{h}$ is an altering distance function and $\phi $ \ is a continuous function such that $\mathbf{h}\left( t\right) \leq t$, $\forall$ $t\geq 0$, $\phi \left( t\right) =0$ if and only if $t=0$, and for each $x,y\in C\left(\Delta,\mathbb{R}\right)$, we have
\begin{eqnarray*}
\left\vert \mathfrak{W}\left( t,s,x\left( s\right) \right) \right\vert +\left\vert
\mathfrak{W}\left( t,s,y\left( s\right) \right) \right\vert \leq \mathbf{h}\left( \left\vert x\left( s\right) \right\vert +\left\vert y\left( s\right) \right\vert \right)  &&-\frac{\phi \left( \underset{t\in \Delta }{\sup }\left\vert x\left( t\right) \right\vert +\underset{t\in \Delta }{\sup }\left\vert y\left(
t\right) \right\vert \right) -2\left\vert \varphi \left( t\right)\right\vert }{A}
\end{eqnarray*}
$\forall$ $t,s\in \Delta $ and $A=\dfrac{\Gamma \left( 1-\nu \left( 1-\mu \right) \right) \Gamma \left( 1+2\nu \left( 1-\mu \right) +\mu \right) }{\Gamma \left( 1+\nu \left( 1-\mu \right) \right) \left( \rm\Psi \left( b\right) -\rm\Psi \left( a\right) \right) ^{2-\mu }}$.

Then the nonlinear $\rm\Psi-$fractional ${\rm VIE}$ {\rm Eq.(\ref{cont51})}, has a unique solution. Moreover, for each $x_{0}\in C\left(\Delta,\mathbb{R}\right) $ the Picard iteration $\left\{ x_{n}\right\} $, which is defined by
\begin{equation*}
x_{n}\left( t\right) =\varphi \left( t\right) +\mathbf{I}_{a+}^{\mu ,\nu ;\rm\Psi }\left[ \mathfrak{W}\left(t,s,x_{n-1}\left(s\right) \right) \right] 
\end{equation*}
$\forall$ $n\in \mathbb{N}$, converges to a unique solution of the nonlinear $\rm\Psi-$fractional ${\rm VIE}$ {\rm Eq.(\ref{cont51})}.
\end{theorem}

\begin{proof} Let $\Lambda=C(\Delta,\mathbb{R}) $. First, note that, $\Lambda$ with the metric $d:\Lambda\times \Lambda\rightarrow \Delta_{1}$ given by
\begin{equation*}
d\left( x,y\right) =\underset{t\in \Delta }{\sup }\left\vert x\left( t\right) -y\left( t\right) \right\vert 
\end{equation*}
$\forall$ $x,y\in \Lambda$, is a complete metric space. Now, we define the function $\mathbf{d}:\Lambda\times \Lambda\rightarrow \Delta_{1} $ by
\begin{equation*}
\mathbf{d}\left( x,y\right) =\underset{t\in \Delta }{\sup }\left\vert x\left( t\right) \right\vert +\underset{t\in \Delta }{\sup } \left\vert y\left( t\right) \right\vert 
\end{equation*}
$\forall$ $x,y\in \Lambda.$ Clearly, $\mathbf{q}$ is a $\beta -$distance on $\Lambda$ and a distance of $d$. We show that $f$ is contractive condition Eq.(\ref{eq24}). Assume that $x,y\in \Lambda$, $t\in \Delta $ and applying $\delta =1$ in the Lemma \ref{lema4}, then we have
\begin{eqnarray*}
\left\vert \left( fx\right) \left( t\right) \right\vert +\left\vert \left(
fy\right) \left( t\right) \right\vert &\leq &\left\vert \varphi \left(t\right) \right\vert +\left\vert \mathbf{I}_{a+}^{\mu ,\nu ;\rm\Psi }\mathfrak{W}\left( t,s,x\left( s\right) \right) \right\vert +\left\vert \varphi \left( t\right) \right\vert +\left\vert \mathbf{I}%
_{a+}^{\mu ,\nu ;\rm\Psi }\mathfrak{W}\left( t,s,y\left( s\right) \right) \right\vert 
\notag \\
&\leq &2\left\vert \varphi \left( t\right) \right\vert +\mathbf{I}%
_{a+}^{\mu ,\nu ;\rm\Psi }\left[ \frac{\mathbf{h}\left( \left\vert
x\left( t\right) \right\vert +\left\vert y\left( t\right) \right\vert
\right) -\phi \left( \mathbf{d}\left( x,y\right) \right) -2\left\vert \varphi \left(
t\right) \right\vert }{A}\right]   \notag \\
&\leq &2\left\vert \varphi \left( t\right) \right\vert +\left[ \frac{%
\mathbf{h}\left( \mathbf{d}\left( x,y\right) \right) -\phi \left( \mathbf{d}\left(
x,y\right) \right) -2\left\vert \varphi \left( t\right) \right\vert }{A}%
\right] \mathbf{I}_{a+}^{\mu ,\nu ;\rm\Psi }(1)  \notag \\
&\leq &\mathbf{h}\left( \mathbf{d}\left( x,y\right) \right) -\phi \left(
\mathbf{d}\left( x,y\right) \right) 
\end{eqnarray*}
$\forall$ $x,y\in \Lambda$. It follows that $f$ satisfies the condition Eq.(\ref{eq24}). Therefore, all conditions of Theorem \ref{teo22} are satisfied and thus $f$ has a unique fixed point. Thus, we conclude the proof.
\end{proof}

\begin{theorem} Consider the nonlinear $\rm\Psi-$fractional ${\rm VIE}$
\begin{equation} \label{cont52}
x\left( t\right) =\phi \left( t\right) +\mathbf{I}_{a+}^{\mu ,\nu ;\rm\Psi }\left[ \mathfrak{W}\left( t,s,x\left( s\right) \right) \right] 
\end{equation}
where $0<\mu <1,$ $0\leq \nu \leq 1$, $a,b\in \mathbb{R} $ with $a<b,$ $\phi :\Delta \rightarrow \mathbb{R}$ and $\mathfrak{W}:\Delta ^{2}\times \mathbb{R}\rightarrow \mathbb{R}$ are two given function. Suppose that $\phi $, $\mathfrak{W}$ are continuous functions and the function $f:C\left(\Delta,\mathbb{R}\right)\rightarrow C\left(\Delta,\mathbb{R}\right)$ defined by
\begin{equation*}
\left( fx\right) \left( t\right) =\phi \left( t\right) +\mathbf{I}_{a+}^{\mu ,\nu ;\rm\Psi }\left[ \mathfrak{W}\left( t,s,x\left( s\right) \right) \right] 
\end{equation*}
$\forall$ $x\in C\left(\Delta,\mathbb{R}\right) $ and $t\in \Delta $ is a continuous function. Besides that, there is $\mathbf{h}\in \Omega $ such that $\mathbf{h}\left( t\right) \leq t$, $\forall$ $t\geq 1$ and for each $x,y\in C\left(\Delta,\mathbb{R}\right)$, we have
\begin{equation*}
\left\vert \mathfrak{W}\left( t,s,x\left( s\right) \right) \right\vert +\left\vert \mathfrak{W}\left( t,s,y\left( s\right) \right) \right\vert \leq \frac{\left[ \mathbf{h}\left( \left\vert x\left( s\right) \right\vert +\left\vert y\left(s\right) \right\vert \right) \right] ^{k_{1}}-2\left\vert \phi \left( t\right) \right\vert }{A}
\end{equation*}
$\forall$ $t,s\in \Delta $, $k_{1}\in \left[ 0,1\right) $ and $A=\dfrac{\Gamma \left( 1-\nu \left( 1-\mu \right) \right) \Gamma \left( 1+2\nu \left( 1-\mu \right) +\mu \right) }{\Gamma\left( 1+\nu \left( 1-\mu \right) \right) \left( \rm\Psi \left( b\right) -\rm\Psi \left( a\right) \right) ^{2-\mu }}$.

Then, the nonlinear $\rm\Psi-$fractional ${\rm VIE}$ {\rm Eq.(\ref{cont52})}, has a unique solution. Moreover, for each $x_{0}\in C(\Delta,\mathbb{R}) $ the Picard iteration $\left\{ x_{n}\right\}$, which is defined by
\begin{equation*}
x_{n}\left( t\right) =\phi \left( t\right) +\mathbf{I}_{a+}^{\mu ,\nu ;\rm\Psi } \left[ \mathfrak{W}\left(t,s,x_{n-1}\left( s\right) \right) \right] 
\end{equation*}
$\forall$ $n\in \mathbb{N}$, converges to a unique solution of the nonlinear $\rm\Psi-$fractional ${\rm VIE}$ {\rm Eq.(\ref{cont52})}.
\end{theorem}

\begin{proof} Let $\Lambda=C\left(\Delta,\mathbb{R}\right) $. First, note that, $\Lambda$ with the metric $d:\Lambda\times \Lambda\rightarrow \Delta_{1} $ given by
\begin{equation*}
d\left( x,y\right) =\underset{t\in \Delta }{\sup }\left\vert x\left( t\right) -y\left( t\right) \right\vert 
\end{equation*}
$\forall$ $x,y\in \Lambda$, is a complete metric space. Now, consider the function $\mathbf{d}:\Lambda\times \Lambda\rightarrow \Delta_{1} $ by
\begin{equation*}
\mathbf{d}\left( x,y\right) =\underset{t\in \Delta }{\sup }\left\vert x\left( t\right) \right\vert +\underset{t\in \Delta }{\sup } \left\vert y\left( t\right) \right\vert 
\end{equation*}
$\forall$ $x,y\in \Lambda$. Clearly, $q$ is a $\beta-$distance on $\Lambda$ and a distance of $d$. We show that $f$ is contractive condition {\rm Eq.(\ref{eq25})}. Suppose that $x,y\in \Lambda$, $t\in \Delta $ and applying $\delta =1$ in the Lemma \ref{lema4}, then we have
\begin{eqnarray*}
\left\vert \left( fx\right) \left( t\right) \right\vert +\left\vert \left(
fy\right) \left(t\right) \right\vert &\leq &\left\vert \phi \left(t\right) \right\vert +\left\vert \mathbf{I}_{a+}^{\mu ,\nu ;\rm\Psi }\mathfrak{W}\left( t,s,x\left( s\right) \right) \right\vert+\left\vert \phi \left(t \right) \right\vert +\left\vert \mathbf{I}%
_{a+}^{\mu ,\nu ;\rm\Psi }\mathfrak{W}\left( t,s,y\left( s\right) \right) \right\vert 
\notag \\
&\leq &2\left\vert \phi \left( t\right) \right\vert +\mathbf{I}_{a+}^{\mu
,\nu ;\rm\Psi }\left[ \frac{\left[ \mathbf{h}\left( \left\vert
x\left( t\right) \right\vert +\left\vert y\left(t \right) \right\vert
\right) \right] ^{k_{1}}-2\left\vert \phi \left( t\right) \right\vert }{A}%
\right]   \notag \\
&\leq &2\left\vert \phi \left( t\right) \right\vert +\left[ \frac{\left[ 
\mathbf{h}\left( \mathbf{d}\left( x,y\right) \right) \right]
^{k_{1}}-2\left\vert \varphi \left(t \right) \right\vert }{A}\right] \mathbf{
I}_{a+}^{\mu ,\nu ;\rm\Psi }(1)   \notag \\
&\leq &\left[ \mathbf{h}\left( \mathbf{d}\left( x,y\right) \right) \right]
^{k_{1}}
\end{eqnarray*}
$\forall$ $x,y\in \Lambda$. 

This implies that 
\begin{equation*}
\underset{t\in \Delta }{\sup }\left\vert \left( fx\right) \left( t\right) \right\vert +\underset{t\in \Delta }{\sup }\left\vert \left( fy\right) \left( t\right) \right\vert \leq \left[ \mathbf{h}\left( \mathbf{d}\left( x,y\right) \right) \right] ^{k_{1}}
\end{equation*}
\textbf{and so}
\begin{equation*}
\mathbf{d}\left( fx,fy\right) \leq \left[ \mathbf{h}\left( \mathbf{d}\left( x,y\right) \right) \right] ^{k_{1}}
\end{equation*}
$\forall$ $x,y\in \Lambda$. Hence, we have

\begin{equation*}
\mathbf{h}\left( \mathbf{d}\left( fx,fy\right) \right) \leq \left[ \mathbf{h}\left( \mathbf{d}\left( x,y\right) \right) \right] ^{k_{1}}
\end{equation*}
$\forall$ $x,y\in \Lambda.$ It follows that $f$ satisfies the condition {\rm Eq.(\ref{eq25})}. Therefore, all conditions of Theorem \ref{teo23} are satisfied and thus $f$ has a unique fixed point. Thus, we conclude the proof.
\end{proof}

\begin{theorem} Let the nonlinear $\rm\Psi-$fractional ${\rm VIE}$
\begin{equation}\label{cont53}
x\left( t\right) =\varphi \left( t\right) +\mathbf{I}_{a+}^{\mu ,\nu ;\rm\Psi } \left[ \mathfrak{W}\left( t,s,x\left( s\right) \right) \right] 
\end{equation}
where $0<\mu <1,$ $0\leq \nu \leq 1$, $a,b\in \mathbb{R}$ with $a<b,$ $\phi :\Delta \rightarrow \mathbb{R}$ and $\mathfrak{W}:\Delta ^{2}\times \mathbb{R}\rightarrow \mathbb{R}$ are two given function. Suppose that the following conditions hold:

\begin{enumerate}
\item  $\varphi $ and $\mathfrak{W}$ are continuous functions;

\item the function $f:C\left(\Delta,\mathbb{R}\right) \rightarrow C\left(\Delta,\mathbb{R}\right) $ defined by 
\begin{equation*}
\left( fx\right) \left( t\right) =\varphi \left( t\right) +\mathbf{I}_{a+}^{\mu ,\nu ;\rm\Psi }\left[ \mathfrak{W}\left( t,s,x\left( s\right) \right) \right] 
\end{equation*}
$\forall$ $x\in C\left(\Delta,\mathbb{R}\right) $ and $\ t\in \Delta $ is a continuous function;

\item there are two functions $\mathbf{h},\phi :\Delta_{1} \rightarrow \Delta_{1} $ with $\mathbf{h}$ is a weak altering distance function and $\phi $ is a right upper semi continuous function such that $\mathbf{h}\left( t\right) >t$, $\forall$ $t>0$ and $ \mathbf{h}\left( t\right) <t$, $\forall$ $t\geq 0$, $\phi \left( 0\right) =0$, for each $x,y\in C\left(\Delta,\mathbb{R}\right)$, we have
\begin{equation*}
\left\vert \mathfrak{W}\left( t,s,x\left( s\right) \right) \right\vert +\left\vert \mathfrak{W}\left( t,s,y\left( s\right) \right) \right\vert \leq \frac{\phi \left( \left\vert x\left( s\right) \right\vert +\left\vert y\left( s\right) \right\vert \right) -2\left\vert \varphi \left( t\right) \right\vert }{A}
\end{equation*}
$\forall$ $t,s\in \Delta $ and $A=\dfrac{\Gamma \left( 1-\nu \left( 1-\mu \right) \right) \Gamma \left( 1+2\nu \left( 1-\mu \right) +\mu \right) }{\Gamma \left( 1+\nu \left( 1-\mu \right) \right) \left( \rm\Psi \left( b\right) -\rm\Psi \left( a\right) \right) ^{2-\mu }}$.
\end{enumerate}

Then the nonlinear $\rm\Psi-$fractional ${\rm VIE}$ {\rm Eq.(\ref{cont53})}, has a unique solution. Moreover, for each $x_{0}\in C\left(\Delta,\mathbb{R}\right)$ the Picard iteration $\left\{ x_{n}\right\} $, which is defined by
\begin{equation}
x_{n}\left( t\right) =\varphi \left( t\right) +\mathbf{I}_{a+}^{\mu ,\nu ;\rm\Psi } \left[ \mathfrak{W}\left( t,s,x_{n-1}\left( s\right) \right) \right] 
\end{equation}
$\forall$ $n\in \mathbb{N}$, converges to a unique solution of the nonlinear $\rm\Psi-$fractional ${\rm VIE}$ {\rm Eq.(\ref{cont53})}.
\end{theorem}

\begin{proof} Let $\Lambda=C\left(\Delta,\mathbb{R}\right) $. First, note that, $\Lambda$ with the metric $d:\Lambda\times \Lambda\rightarrow \Delta_{1} $ given by
\begin{equation*}
d\left( x,y\right) =\underset{t\in \left[ a,b\right] }{\sup }\left\vert x\left( t\right) -y\left( t\right) \right\vert 
\end{equation*}
$\forall$ $x,y\in \Lambda$, is a complete metric space. Now, we consider $\mathbf{d}:\Lambda\times \Lambda\rightarrow \Delta_{1} $ by
\begin{equation*}
\mathbf{d}\left( x,y\right) =\underset{t\in \left[ a,b\right] }{\sup }\left\vert x\left( t\right) \right\vert +\underset{t\in \Delta }{\sup } \left\vert y\left( t\right) \right\vert 
\end{equation*}
$\forall$ $x,y\in \Lambda$. Clearly, $\mathbf{d}$ is a $\beta -$distance on $\Lambda$ and a distance of $d$. We show that $f$ is contractive condition Eq.(\ref{eq23}). Assume that $x,y\in \Lambda$, $t\in \Delta $ and applying $\delta =1$ in the Lemma \ref{lema4}, then we have
\begin{eqnarray*}
\left\vert \left( fx\right) \left( t\right) \right\vert +\left\vert \left( fy\right) \left( t\right) \right\vert &\leq &\left\vert \varphi \left( t\right) \right\vert +\left\vert \mathbf{I}_{a+}^{\mu ,\nu ;\rm\Psi }\mathfrak{W}\left( t,s,x\left( s\right) \right) \right\vert +\left\vert \varphi \left( t\right) \right\vert +\left\vert \mathbf{I}_{a+}^{\mu ,\nu ;\rm\Psi }\mathfrak{W}\left( t,s,y\left( s\right) \right) \right\vert 
\notag \\
&\leq &2\left\vert \varphi \left( t\right) \right\vert +\mathbf{I}_{a+}^{\mu ,\nu ;\rm\Psi }\left[ \frac{\phi \left( \left\vert x\left(t\right) \right\vert +\left\vert y\left( t\right) \right\vert \right)-2\left\vert \varphi \left( t\right) \right\vert }{A}\right]   \notag \\
&\leq &2\left\vert \varphi \left( t\right) \right\vert +\left[ \frac{\phi
\left( \mathbf{d}\left( x,y\right) \right) -2\left\vert \varphi \left( t\right)
\right\vert }{A}\right] \mathbf{I}_{a+}^{\mu ,\nu ;\rm\Psi }(1)
\notag \\
&\leq &\phi \left( q\left( x,y\right) \right) 
\end{eqnarray*}
$\forall$ $x,y\in \Lambda$. 

This implies that 
\begin{equation*}
\underset{t\in \Delta }{\sup }\left\vert \left( fx\right) \left(t\right) \right\vert +\underset{t\in \Delta }{\sup }\left\vert \left( fy\right) \left( t\right) \right\vert \leq \phi \left( \mathbf{d}\left(x,y\right) \right) 
\end{equation*} 
and so
\begin{equation*}
\mathbf{d}\left( fx,fy\right) \leq \phi \left( \mathbf{d}\left( x,y\right) \right) 
\end{equation*}
$\forall$ $x,y\in X$. Hence, we have
\begin{equation*}
\mathbf{h}\left( \mathbf{d}\left( fx,fy\right) \right) \leq \phi \left( \mathbf{d}\left( x,y\right) \right) 
\end{equation*}
$\forall$ $x,y\in \Lambda.$ If follows that $f$ satisfies the condition Eq.(\ref{eq23}). Thus, by Theorem \ref{teo24}, $f$ has a unique fixed point.
\end{proof}

The results presented above are all directed to the $\rm\Psi-$fractional integral. Next, we will briefly discuss population growth through the ${\rm \Psi-H}$ fractional derivative.

Let $N(t)$ be the number of people of a population in time $t$, and $B$ and $M$ the birth rate and the mortality rate, respectively. The population growth rate is given by the differential equation 
\begin{equation}  \label{MAL}
N^{\prime }\left( t\right) =\left( B-M\right) N\left( t\right) =\lambda
N\left( t\right),
\end{equation}
where $\lambda$ is the population growth rate. Suppose that $B$ and $M$ are constants, then $\lambda$ is constant, so if $N(0)=N_{0}$ is the initial population then the solution for the Cauchy problem satisfies \cite{almeida,RCA} 
\begin{equation}  \label{MAL1}
N\left( t\right) =N_{0}e^{\lambda t}.
\end{equation}

On the other hand, using the $\rm\Psi $-fractional derivative, we introduce the respective fractional differential equation \textrm{Eq.(\ref{MAL})},
given by 
\begin{equation}
^{H}\mathscr{D}_{0+}^{\mu ,\nu ;\rm\Psi }\text{ }N\left( t\right) =\lambda N\left(
t\right) ,  \label{MAL2}
\end{equation}%
with $t\geq 0$, $0<\mu <1$ and $0\leq \nu \leq 1$.

From Lemma \ref{lema6}, the solution of this fractional differential equation \textrm{Eq.(\ref{MAL2})}, together with the initial condition $N(0)=N_{0}$, is the function $$N\left( t\right) =N_{0}\mathbf{E}_{\mu }\left( \lambda \left( \rm\Psi \left( t\right) -\rm\Psi \left( 0\right) ^{\mu }\right) \right). $$
As a particular case, choosing $\rm\Psi (x)=x$, we obtain the fractional derivative in the Hilfer sense \cite{JEM1} on the left and consequently the solution of the problem will be given by $N\left( t\right) =N_{0}\mathbf{E}_{\mu }\left( \lambda \left(t^{\mu }\right) \right) $.

\begin{remark}
Recently, Sousa and Oliveira {\rm \cite{JEM1}} introduced the so-called ${\rm \Psi-H}$ fractional derivative. As a particular case, it is possible to obtain the ${\rm \Psi-C}$ fractional derivative. In addition, Almeida {\rm \cite{RCA}} introduces the ${\rm \Psi-C}$ fractional derivative discuss some functions, among them, one involving the Malthusian model. The solution of the Malthusian fractional model, can be introduced by ${\rm \Psi-H}$, ${\rm \Psi-C}$, ${\rm \Psi-RL}$ fractional derivatives {\rm \cite{AHMJ} } and now by the $\rm\Psi-$fractional derivative on the left, is always the same, as proved in {\rm Lemma \ref{lema6}}. In this sense, we will omit the behavior of the solution that refers to population growth, because it behaves similar to the graph of the ${\rm \Psi-C}$ fractional derivative model.
\end{remark}

\section{Concluding remarks}

\textbf{The main purpose of this paper was to introduce a new fractional integral the so-called $\rm\Psi-$fractional integral through the ${\rm \Psi-RL}$ fractional integral. Some results were discussed and the observations were presented, highlighting the relation with the fractional operator proposed by Xu and Agrawal \cite{xu}. We presented a brief study about the uniqueness of solutions of the nonlinear $\rm\Psi-$fractional ${\rm VIE}$ via $\beta-$distance and $\beta-$generalized $\mathbf{SS}$-contraction function. We also saw the importance of ${\rm M-L}$ functions in the solution of fractional differential equations, particularly population growth. Note the importance of these new operators, which are related to another function, i.e. $\rm\Psi(\cdot)$, allowing the choice and consequently obtain a fractional operator with the characteristics presented in both definitions.}

We conclude that, from an example involving the one parameter ${\rm M-L}$ function, generalized fractional operators with the freedom of choice of the $\rm\Psi$ function, allows results closer to the reality. 

\textbf{Is important to note that some points raised during the development of the paper are different from the Riemann-Liouville fractional integral, as below:}
\begin{itemize}
\item \textbf{The proper definition of the $\rm\Psi-$fractional integral containing the ${\rm \Psi-RL}$ fractional derivative and the integral $I_{a+}^{1;\rm\Psi }\left( \cdot \right)$;}

\item \textbf{The $\rm\Psi-$fractional integral $\mathbf{I}_{a+}^{\mu ,\nu ;\rm\Psi }\left( \cdot\right) $ as a particular case, of the Riemann-Liouville integral in the limit $\mu \rightarrow 1$;}

\item \textbf{To obtain the identity operator, working with the Riemann-Liouville fractional integral, we take $\mu \rightarrow 0$. For $\rm\Psi-$fractional integral $\mathbf{I}_{a+}^{\mu ,\nu ;\rm\Psi }\left( \cdot \right) $, besides applying the limit $\mu \rightarrow 0,$ we also have to take the limit with $\nu \rightarrow 1$;}

\item \textbf{Also, we do not resort the same function$^{H}\mathscr{D}_{a+}^{\mu ,\nu ;\rm\Psi }\text{ }\mathbf{I}_{a+}^{\mu ,\nu ;\rm\Psi }f\left( \cdot \right) \neq f\left( \cdot \right)$.}

\end{itemize}

\textbf{One of the main differences between the Riemann-Liouville fractional integral and the $\rm\Psi-$fractional integral $\mathbf{I}_{a+}^{\mu ,\nu ;\rm\Psi }\left( \cdot \right) $, is the calculation of the Laplace transform of both integrals. For Riemann-Liouville, we know that $\mathscr{L}\left( \mathcal{I}_{a+}^{\mu }f\left( t\right) \right) =\dfrac{F\left( s\right) }{s^{\mu }},$ where $F\left( s\right) =\mathscr{L}\left( f\left( t\right) \right) $ is the Laplace transform. On the other hand, it is not possible to calculate the Laplace transform of the $\rm\Psi-$fractional integral $\mathbf{I}_{a+}^{\mu ,\nu ;\rm\Psi }\left( \cdot \right) $, since it does not yet have a Laplace transform with respect to another function $\rm\Psi$. The same argument is also valid for the Fourier transform. In this sense, the definition proposed here is in fact distinct from the classical definition of Riemann-Liouville fractional integral.}

As future works, one can think the possibility of proposing fractional operators of variable order (one or two) \cite{almeida10,samko} and realize functions, such as studying existence, uniqueness, stability and continuous dependence of Cauchy type problems \cite{sousa,furati,zhou,peng}.

\section*{Acknowledgment}

\textbf{We are grateful to the editor and anonymous referee for the suggestions that have improved the manuscript.}

\bibliography{ref}
\bibliographystyle{plain}

\end{document}